\pdfoutput=1
\documentclass[letterpaper,leqno,11pt]{article}

\usepackage[margin=1.5in]{geometry}
\usepackage[english]{babel}
\usepackage{csquotes}
\usepackage{amsthm,amsmath}
\usepackage{amsfonts}
\usepackage{hyperref}
\hypersetup{pdfencoding=auto,colorlinks=true,linkcolor=black,citecolor=black}

\usepackage{arbormath}
\DeclareThmEnv
\appto\remark{\leftskip\parindent}

\usepackage[inline]{enumitem}
\usepackage{tikz}

\usepackage{microtype}

\usepackage[style=authoryear,backend=biber,mincrossrefs=999]{biblatex}
\DeclareNameAlias{sortname}{first-last}
\AtBeginBibliography{%
    \renewcommand*\mkbibnamelast[1]{\textsc{#1}}
    \renewcommand*\mkbibnamefirst[1]{\textsc{#1}}
    \renewcommand*\mkbibnameprefix[1]{\textsc{#1}}
    \renewcommand*\mkbibnameaffix[1]{\textsc{#1}}}
\DeclareFieldFormat{eprint:mr}{%
  \mkbibacro{MR}\addcolon\addnbspace
  \ifhyperref
    {\href{http://www.ams.org/mathscinet-getitem?mr=#1}{\nolinkurl{#1}}}
    {\nolinkurl{#1}}}

\renewcommand\Pr{\operatorname{P}}

\AtBeginDocument{}
\AtBeginDocument{}

\newcommand\blfootnote[1]{%
  \begingroup\renewcommand\thefootnote\relax\footnotetext{#1}%
  \addtocounter{footnote}{-1}\endgroup}

\begin{document}
\def\itemautorefname~{}
\newcommand\todo[1]{}

\stepcounter{footnote}
\stepcounter{footnote}
\edef\authum{\thefootnote}
\stepcounter{footnote}
\edef\authmr{\thefootnote}
\stepcounter{footnote}
\edef\authgo{\thefootnote}
\setcounter{footnote}{0}
\makeatletter
\newcommand\authmark[1]{\textsuperscript{\@fnsymbol#1}}
\makeatother

\title{Shortest Path through Random Points}
\author{Sung Jin Hwang\authmark\authum\authmark\authgo \and Steven B. Damelin\authmark\authmr \and Alfred O. Hero III\authmark\authum}
\date{\authmark\authum University of Michigan\\ \authmark\authmr Mathematical Reviews}
\maketitle
\blfootnote{\textit{AMS 2000 subject classifications:} Primary 60F15; secondary 60C05, 53B21}
\blfootnote{\textit{Keywords and phrases:} shortest path, power-weighted graph, Riemannian geometry, conformal metric}
\blfootnote{\authmark\authum This work was partially supported by ARO grant W911NF-09-1-0310 and NSF grant CCF-1217880.}
\blfootnote{\authmark\authmr This work was supported, in part, by American Mathematical Society and grant: \enquote{Computational Research Initiative in Imaging and Remote Sensing-8-2011,} Center for High Performance Computing, South Africa.}
\blfootnote{\authmark\authgo Currently at Google Inc.}


\begin{abstract}
\ Let $(M,g_1)$ be a complete $d$-dimensional Riemannian manifold for $d > 1$.  Let $\cX_n$ be a set of $n$ sample points in $M$ drawn randomly from a smooth Lebesgue density $f$ supported in $M$.  Let $x,y$ be two points in $M$.  We prove that the normalized length of the power-weighted shortest path between $x,y$ through $\cX_n$ converges almost surely to a constant multiple of the Riemannian distance between $x,y$ under the metric tensor $g_p = f^{2(1-p)/d} g_1$, where $p > 1$ is the power parameter.  
\end{abstract}

\section{Introduction}
The shortest path problem (see e.g., \textcite{Cormen:2009a,Dijkstra:1959a}) is of interest both in theory and in applications since it naturally arises in combinatorial optimization problems, such as optimal routing in communication networks, and efficient algorithms exist to solve the problem.  In this paper, we are interested in the shortest paths over random sample points embedded in Euclidean and Riemannian spaces.

Many graph structures over Euclidean sample points have been studied in the context of Beardwood-Halton-Hammersley (BHH) theorem and its extensions.  The BHH theorem states that the law of large numbers (LLN) holds for certain spanning graphs over random samples.  Such graph structures include the travelling salesman path (TSP), the minimal spanning tree (MST), and the nearest neighbor graphs ($k$-NNG).  See \textcite{Steele:1997a} and \textcite{Yukich:1998a}.  The BHH theorem applies to graphs that span all of the points in the random sample.  This paper establishes a BHH-type theorem for shortest paths between any two points.

In the last few years, the asymptotic theory for spanning graphs such as the MST, the $k$-NNG, and the TSP has been extended to the Riemannian case, e.g., \textcite{Costa:2004a} extended the MST asymptotics in the context of entropy and intrinsic dimensionality estimation.  More general non-Euclidean extensions have been established by \textcite{Penrose:2011a}.  This paper extends the BHH theorem in a different direction: the shortest path between random points in a Riemannian manifold.

The asymptotic properties of paths through random Euclidean sample points have been studied mainly in first-passage percolation (FPP) models \autocite{Hammersley:1966a}.  Shortest paths have been studied in FPP models in the context of first passage time or travel time with lattice models \autocite{Kesten:1987a} or (homogeneous) continuum models \autocite{Howard:1997a}.  Under the FPP lattice model, \textcite{LaGatta:2010a} extended these results to the non-Euclidean case where interpoint distances are determined by a translation-invariant random Riemannian metric in $\reals^d$.  This paper makes a contribution in a different direction.  We assume a non-homogeneous continuum model and establish convergence of the shortest path lengths to density-dependent deformed Riemannian distances. The convergent limit reduces to the result of \textcite{Howard:2001a} when specialized to a homogeneous Euclidean continuum model.

\section{Main results}
In this paper, a smooth function is an infinitely differentiable function, i.e., $f \in \Cinfty$.
A smooth manifold means its transition maps are smooth.

Let $(M,g_1)$ be a smooth $d$-dimensional Riemannian manifold without boundary with Riemannian metric tensor $g_1$ and $d > 1$.
Recall that a Riemannian metric tensor, often simply called a Riemannian metric, on a manifold is a family of positive definite inner products on the tangent spaces of the manifold.
When $M = \reals^d$, $g_1$ is the standard Euclidean inner product.
The use of the subscript on $g_1$ will become clear shortly.

Consider a probability space $(M,\BbbB,\Pr)$ where $\Pr$ is a probability distribution over Borel subsets $\BbbB$ of the sample space $M$.  Assume that the distribution has a Lebesgue probability density function (pdf) $f$ with respect to $g_1$.  Let $X_1,X_2,\dots$ denote an i.i.d.\ sequence drawn from this density and let the first $n$ samples from this sequence be denoted by $\cX_n = \braces{X_1,\dots,X_n}$.  The sequence $\cX_n$ will be associated with the nodes in a undirected simple graph whose edges have weight equal to the power weighted Euclidean distance between pairs of nodes.  We will use indexing by $n$ of a generic non-random point $x_n \in M$.  This point is not related in any way to the random variable $X_n$.  For realizations, we will use the notation $X_n(\omega)$ where $\omega$ is an elementary outcome in the sample space.

For $p > 1$, called the power parameter, define a new metric tensor $g_p = f^{2(1-p)/d} g_1$.  That is, if $Z_x$ and $W_x$ are two tangent vectors at a point $x \in M$, then $g_p(Z_x,W_x) = f(x)^{2(1-p)/d} g_1(Z_x,W_x)$.  The deformed metric tensor $g_p$ is well-defined for every $x$ with $f(x) > 0$, and $g_p$ is a Riemannian metric tensor when $f \in \Cinfty$.  In this paper, we assume $p > 1$ except for a few places where we compare with the undeformed case $p = 1$.

The main result of this paper, stated as \autoref{thm:main}, establishes an asymptotic limit of the lengths of the shortest paths through locally finite point processes.  A subset $A \subset M$ is locally finite if $A \cap B$ is finite for every $B \subset M$ of finite volume.  For example, a homogeneous Poisson process in $\reals^d$ is locally finite with probability one.  For $x,y \in M$ and locally finite $A \subset M$, let $L(x,y; A)$ denote the power-weighted shortest path length from $x$ to $y$ through $A \cup \braces{x,y}$.  Let the edge weight between two points $u$ and $v$ be defined as $\dist_1(u,v)^p$ where $\dist_1$ denotes the Riemannian distance under the metric tensor $g_1$.
A path $\pi$ through points $x_0,\dots,x_k$ has power-weighted length $\sum_{i=0}^{k-1} \dist_1(x_i,x_{i+1})^p$.

For $x \in M$ and $r > 0$, we denote by $B(x; r)$ the open ball in $M$ of radius $r$ centered at $x$, i.e., $B(x; r) = \braces{u \in M \colon \dist_1(x,u) < r}$.

\subsection{Main result}
Let $\dist_p$ denote the deformed distance under $g_p$,
\begin{equation}\label{eq:distp_def}
    \dist_p(x,y) = \inf_\gamma \int_0^1 f(\gamma_t)^{(1-p)/d} \sqrt{g_1(\gamma_t', \gamma_t')} \,dt,
\end{equation}
where the infimum is taken over all piecewise smooth curves $\gamma\colon [0,1] \to M$ such that $\gamma_0 = x$ and $\gamma_1 = y$.  When a curve achieves the infimum, we call the curve a $g_p$-geodesic.

The following is the main result of this paper.

\begin{theorem}\label{thm:main}
Assume that $M$ is compact, and that $f$ is continuous with $\inf_M f > 0$.
There exists a constant $C(d,p) > 0$, which only depends on $d$ and $p$, satisfying the following.
Let $b > 0$ and $\vareps > 0$.
Then there exists $\theta_0 > 0$ such that
\begin{equation*}
  \Pr \paren*{ \sup_{x,y} \abs*{ \frac{L(x,y; \cX_n)}{n^{(1-p)/d} \dist_p(x,y)} - C(d,p) } > \vareps }
  \leq \exp\paren[1]{ -\theta_0 n^{1/(d+2p)} },
\end{equation*}
for all sufficiently large $n$, where the supremum is taken over $x,y \in M$ with $\dist_1(x,y) \geq b$.
\end{theorem}

The constant $C(d,p)$ is fixed throughout this paper (This is the same constant that is denoted as $\mu$ in \textcite{Howard:1997a,Howard:2001a}).
When $p = 1$, there is no power-weighting of the edges, and $C(d,1) = 1$.

%
The requirement $\dist_1(x,y) \geq b > 0$ can be relaxed.
The probability upper bound $\exp(-\theta_0 n^{1/(d+2p)})$ can be written as $\exp(-\theta_0^\prime (n r_n^d)^{1/(d+2p)} + O(\log n))$ where $\theta_0^\prime = \theta_0 b^{-d/(d+2p)}$ and where $x,y$ is constrained to satisfy $\dist_1(x,y) \geq r_n$ for some positive sequence $r_n$.
Therefore for the probability upper bound to be non-trivial, $n r_n^d / \log n$ must go to infinity.
The separation requirement $\dist_1(x,y) \geq b > 0$ is one sufficient condition that ensures this property.
%

A similar convergence result holds when $M$ is complete, but not necessarily compact, giving the almost sure limit stated below.

\begin{theorem}\label{thm:main-complete}
Assume that $M$ is complete and that $f$ is continuous with $f(u) > 0$ for all $u \in M$.  Fix $x,y \in M$.  Then
\begin{equation*}
    \lim_{n\to\infty} n^{(p-1)/d} L(x,y; \cX_n) = C(d,p) \dist_p(x,y) \quad\text{a.s.}
\end{equation*}
The constant $C(d,p)$ is the same constant as in \autoref{thm:main}.
\end{theorem}

\newtheorem*{remark*}{Remark}
\begin{remark*}
In the case where the pdf $f \in \Cinfty$, then the deformed metric tensor $g_p$ is a Riemannian metric tensor, and $\dist_p$ is a Riemannian distance.  \autoref{thm:main} and \autoref{thm:main-complete} connect an algorithmic quantity, power-weighted shortest path lengths, to a geometric quantity, Riemannian distances.
\end{remark*}

\subsection{Discussion}
We use shorthand notation $\cL_\lambda(x,y)$ to denote $L(x,y; \cH_\lambda)$ where $\cH_\lambda$ is a homogeneous Poisson point process of intensity $\lambda > 0$ in $\reals^d$.

\autoref{thm:main} and \autoref{thm:main-complete} can be compared to analogous results in the continuum FPP model of \textcite{Howard:2001a}.  The main differences are the following:  \begin{enumerate*}[label=(\roman*)] \item the results of \textcite{Howard:2001a} are restricted to the case of uniformly distributed node locations $\cH_\lambda$ in Euclidean spaces while our results also hold for the case of non-uniformly distributed points in compact or complete manifolds; \item our convergence rates improve upon those of \textcite{Howard:2001a}.\end{enumerate*}

Specifically, \textcite[Theorem~2.2]{Howard:2001a} establish a bound on the shortest path lengths in a homogeneous Poisson point process.  Recall that $\cL_\lambda(x,y)$ denotes the power-weighted shortest path length from $x \in \reals^d$ to $y \in \reals^d$ through random nodes in a homogeneous Poisson point process $\cH_\lambda$ of intensity $\lambda > 0$.

\textcite[Theorem~2.2]{Howard:2001a} state the following.  Let $\kappa_1 = \min(1,d/p)$, $\kappa_2 = 1/(4p+3)$, and $e_1 \in \reals^d$ be a unit vector.  For any $0 < b < \kappa_2$, there exists a constant $C_0$ (depending on $b$) such that for $t > 0$ and $t^b \leq s \sqrt{t} \leq t^{\kappa_2 - b}$,
\begin{equation}\label{eq:howard-ineq1}
    \Pr \paren*{\abs*{\frac1t \cL_1(0,te_1) - C(d,p)} > s}
    \leq \exp\paren[1]{-C_0 (s \sqrt{t})^{\kappa_1}}.
\end{equation}

Note that the bound in \eqref{eq:howard-ineq1} decays to zero no faster than $\exp(-C_0 t^{\kappa_1\kappa_2})$ where $\kappa_1\kappa_2 = t^{\min(1,d/p)/(4p+3)}$.

On the other hand, for arbitrary (uniform or non-uniform) density, our \autoref{thm:euc-unif-local} implies, after simple Poissonization of the sequence $\cX_n$ that there exists some $\theta > 0$ such that (see the appendix)
\begin{equation}\label{eq:main-cor}
  \Pr \paren*{ \abs*{ \frac1t \cL_1(0,te_1) - C(d,p) } > s }
  \leq \exp\paren[1]{-\theta t^{d/(d+2p)}}
\end{equation}
for all sufficiently large $t$.
Therefore the decay is exponential in $t^{d/(d+2p)}$.
Under the condition $d \geq 1$ and $p > 1$, the decay rate \eqref{eq:main-cor} is faster than the rate \eqref{eq:howard-ineq1}.

It is interesting to compare \autoref{thm:main} with BHH results.  The convergence result established in this paper differs from previous BHH theorems in two ways.  The first difference is that \autoref{thm:main} specifies a limit of the shortest path through $\cX_n$ while BHH theory \autocite{Steele:1997a,Yukich:1998a} specifies limits of the total length of a graph spanning $\cX_n$, e.g., the minimal spanning tree (MST) or the solution to the traveling salesman problem (TSP).  The second difference is that the shortest path has fixed anchor points, hence it is not translation-invariant.  This is in contrast to BHH theory developed in \textcite{Penrose:2003b} and \textcite{Penrose:2011a} where Euclidean functionals are generalized to locally stable functionals while the translation-invariance requirement is maintained.

\section{Main proofs}\label{sec:proofs}
An obvious but important property of $L(x,y; A)$ for $x,y \in M$ and locally finite $A \subset M$ is that if $A^\prime \subset A$ then $L(x,y; A) \leq L(x,y; A^\prime)$.  This property is used in several places in the proofs.

We define the constant $\alpha = 1 / (d+2p)$ that is used throughout the paper.

\subsection{Local convergence results}
\autoref{thm:main} states a convergence result applying to random variables in Riemannian manifolds.  \autoref{thm:main} is obtained by an extension of a simpler theorem on Euclidean space.

We first prove an upper bound for shortest path edge lengths.

\begin{lemma}\label{thm:euc-shortlink}
Let $z \in \reals^d$ and $R > 0$.
Assume that $\cX_n$ is i.i.d.\ in $\reals^d$ with pdf $f$, and that $f_m  = \inf\, \braces{ f(u) \colon u \in B(z; R) }$ is strictly positive.
Fix $b > 0$.

Define the event $H_n(i,j)$ for each pair $1 \leq i \neq j \leq n$ as the intersection of the following events \begin{enumerate}[label=(\roman*)] \item both $X_i$ and $X_j$ are in $B(z; R)$, \item $\abs{X_i - X_j} > b^\alpha (n f_m)^{(\alpha-1)/d}$, and \item the shortest path from $X_i$ to $X_j$ over $\cX_n$ contains no sample point $X_k$ other than $X_i$ and $X_j$.\end{enumerate}

Let $F_n = \bigcap_{i,j} \paren[1]{H_n(i,j)^c}$, where the superscript $c$ denotes set complement.
Then there exists a constant $\theta_1 > 0$ such that
\begin{equation*}
  1 - \Pr(F_n)
  \leq \exp\paren[1]{-\theta_1 n^\alpha}
\end{equation*}
for all sufficiently large $n$.
\end{lemma}

\begin{proof}
Define $h(X_i,X_j; \cdot)\colon \reals^d \to \reals$,
\begin{equation}\label{eq:h-def}
    h(X_i,X_j; u) = \abs{X_i - u}^p + \abs{X_j - u}^p - \abs{X_i-X_j}^p,
\end{equation}
and let $\Theta(X_i,X_j) = \braces{u \in \reals^d \colon h(X_i,X_j; u) < 0}$.  Note that if $X_k \in \Theta(X_i,X_j)$, then $X_i \to X_k \to X_j$ is shorter than $X_i \to X_j$ as measured by the sum of power-weighted edge lengths.  Note that the volume of $\Theta(X_i,X_j)$ is a function of the distance $\abs{X_i-X_j}$ and that a portion of $\Theta(X_i,X_j)$ intersects $B(z; R)$.  Therefore there exists a constant $\theta_1^\prime = \theta_1^\prime(d,p) > 0$ such that the intersection volume is at least $\theta_1^\prime \abs{X_i-X_j}^d$ for all sufficiently large $n$.

Suppose that event $H_n(1,2)$ occurs.  Then the shortest path from $X_1$ to $X_2$ contains no sample point other than $X_1$ and $X_2$, and the intersection of $\Theta(X_1,X_2)$ and $B(z; R)$ cannot contain any of $X_3,X_4,\dots,X_n$.  Since it is assumed that $\abs{X_1-X_2} > b^\alpha (n f_m)^{(\alpha-1)/d}$,
\begin{equation*}
    \Pr\paren[1]{H_n(1,2)}
    \leq \paren[1]{1 - \theta_1^\prime (n f_m b^d)^\alpha / n }^{n-2}.
\end{equation*}
There are $n (n-1) / 2 \leq n^2$ pairs of sample points, hence
\begin{equation*}
    1 - \Pr(F_n)
    = \Pr\paren[1]{\bigcup\nolimits_{i < j} H_n(i,j)}
    \leq n^2 \paren[1]{1 - \theta_1 n^\alpha / n}^{n-2}
\end{equation*}
where $\theta_1 = \theta_1^\prime (f_m b^d)^\alpha$.
\end{proof}

Next we provide results on the number of nodes in the shortest paths (\autoref{thm:Ln-size}), and the mean convergence of $\BbbE L_n$ (\autoref{thm:Ln-mean}).
We will establish these results using the theory of Poisson processes in \autoref{sec:size-mean-proof}.
\autoref{thm:Ln-mean} involves the constant $C(d,p)$ in \autoref{thm:main}.
The definition of this constant will be given in \eqref{eq:cdp-definition} in the proof of \autoref{thm:poisson-mean}.

\begin{lemma}\label{thm:Ln-size}
Let $z \in \reals^d$, $R_2 > R_1 > 0$.
Assume that the pdf $f$ is uniform in $B(z; R_2)$, i.e., $f(u) = f(z) > 0$ for all $u \in B(z; R_2)$, but may have probability mass outside $B(z; R_2)$.
Then there exists a constant $C_* > 0$, depending only on $d$ and $p$, satisfying the following.

For $x,y \in B(z; R_1)$ with $x \neq y$, let $\# L(x,y; \cX_n \cap B(z; R_2))$ denote the number of nodes in the shortest path, and let $G_n(x,y)$ denote the event that
\begin{equation*}
  \frac{ \# L(x,y; \cX_n \cap B(z; R_2)) }{(n f(z))^{1/d} \abs{x-y}} \leq C_*.
\end{equation*}
For any given $b \in (0, 2R_1)$, there exists $\theta_2 > 0$ and $n_0 > 0$ such that if $\abs{x-y} \geq b$ and $n \geq n_0$ then
\begin{equation*}
  1 - \Pr(G_n(x,y))
  \leq \exp\paren[1]{-\theta_2 n^{1/(d+2p-1)}}.
\end{equation*}
\end{lemma}

\begin{proposition}\label{thm:Ln-mean}
Let $z \in \reals^d$, $R_2 > R_1 > 0$.
Assume that the pdf $f$ is uniform in $B(z; R_2)$ but may have probability mass outside $B(z; R_2)$.
Fix $\vareps > 0$ and $b \in (0, 2R_1)$.
Then there exists $n_0 > 0$ such that for all $n \geq n_0$ and $x,y \in B(z; R_1)$ with $\abs{x-y} \geq b$,
\begin{equation*}
  \abs*{\frac{\BbbE L(x,y; \cX_n \cap B(z; R_2))}{(n f(z))^{(1-p)/d} \abs{x-y}} - C(d,p)} < \vareps.
\end{equation*}
\end{proposition}

From \autoref{thm:euc-shortlink}, \autoref{thm:Ln-size}, and \autoref{thm:Ln-mean}, we obtain the following local convergence result.

\begin{proposition}\label{thm:euc-unif}
Let $z \in \reals^d$, $R_2 > R_1 > 0$.
Assume that the pdf $f$ is uniform in $B(z; R_2)$ but may have probability mass outside.
Fix $\vareps > 0$ and $b \in (0, 2R_1)$.
Then there exists a constant $\theta_3 > 0$ such that for all sufficiently large $n$ and for all $x,y \in B(z; R_1)$ with $\abs{x-y} \geq b$,
\begin{equation*}
  \Pr \paren*{\abs*{\frac{L(x,y; \cX_n \cap B(z; R_2))}{(n f(z))^{(1-p)/d} \abs{x-y}} - C(d,p)} > \vareps}
  \leq \exp\paren[1]{-\theta_3 n^\alpha}.
\end{equation*}
\end{proposition}

While it is possible to obtain a weakened form of \autoref{thm:euc-unif} from \textcite{Howard:2001a}, we provide an alternative proof with improved convergence rate.

In the proof of \autoref{thm:euc-unif}, Talagrand's convex distance \autocite[See][Section~4.1]{Talagrand:1995a} is used in the following form.
Let $\omega$ be an elementary outcome in the sample space, and let $A$ be a measurable event with respect to $n$ sample points $X_1,\dots,X_n$.
Define the convex distance $d_c(\omega; A)$ of $\omega$ from $A$,
\begin{equation}\label{eq:convex-distance-def}
    d_c(\omega; A) = \adjustlimits\sup_{s_1,\dots,s_n} \min_{\eta \in A} \sum_i s_i 1_{\braces{X_i(\omega) \neq X_i(\eta)}}
\end{equation}
where the supremum is taken over $s_1,\dots,s_n \in \reals$, $\sum_i \abs{s_i}^2 \leq 1$.
For $t > 0$ define $A_t$ as the \emph{enlargement} of $A$ by $t$,
\begin{equation}\label{eq:talagrand-enlargement}
    A_t = \braces{\omega \colon d_c(\omega; A) \leq t}.
\end{equation}
This notation will be used only in the proof of \autoref{thm:euc-unif}.
Talagrand's concentration inequality \autocite[Theorem~4.1.1]{Talagrand:1995a} is
\begin{equation}\label{eq:talagrand-ineq}
    \Pr(A) (1 - \Pr(A_t)) \leq \exp\paren*{-\frac14 t^2}.
\end{equation}

\begin{proof}[Proof of \autoref{thm:euc-unif}]
Our proof is structured similarly to that of \textcite[Theorem~1.3]{Yukich:2000a} and \textcite[Section~7.1]{Talagrand:1995a}.
For convenience define $\tau_n = b (n f(z))^{1/d}$ and $\zeta_n = \tau_n^\alpha (nf(z))^{-1/d}$.
Let
\begin{itemize}
\item $F_n$ be the event that all the edges of the shortest path have distances at most $\zeta_n$ (See \autoref{thm:euc-shortlink} for $F_n$),
\item $G_n$ be the event that $\# L_n(x_n,y_n) \leq C_* \tau_n$ where the constant $C_*$ is specified in \autoref{thm:Ln-size},
\item $H_n$ be the event that at every point $u \in B(z; R_2)$, at least one of the sample points is in $B(u; \zeta_n)$.
\end{itemize}

All these events occur with high probability.
Both $1 - \Pr (F_n)$ and $1 - \Pr (G_n)$ are exponentially small in $n^\alpha$ by \autoref{thm:euc-shortlink} and \autoref{thm:Ln-size}, respectively.
The probability $1 - \Pr (H_n)$ may be shown to be exponentially small as well by an argument similar to the proof to \autoref{thm:euc-shortlink}, which we will outline here.
Let $\braces{ B(w_i; 2^{-1} \zeta_n, 1 \leq i \leq m }$ be an open cover of $B(z; R_2)$ with $m = O(n)$.
The probability that at least one of the open balls comprising the cover does not contain any sample point is bounded above by $m (1 - 4^{-d} f(z) V_d \zeta_n^d)^n$, where $V_d$ denotes the volume of a unit ball.
This upper bound is exponentially small in $n f(z) \zeta_n^d = (n f(z) b^d)^\alpha$, hence exponentially small in $n^\alpha$ as $n$ goes to infinity.

\newcommand\Ln{\ensuremath{L_n^B}}
We use shorthand notation \Ln\ for $L(x,y; \cX_n \cap B(z; R_2))$.  For $a > 0$, define $W_n(a)$ to be the event that $\Ln \geq a$.  Let $\omega \in F_n \cap G_n$ and $\eta \in H_n \cap W_n(a)$ be two elementary outcomes in the sample space.  If the shortest path $L_n(\omega)$ from $x$ to $y$ through the realization $\cX_n(\omega) = \braces{ X_1(\omega),\dots,X_n(\omega) }$ is the sequence
\begin{equation*}
  x_n = \pi_0(\omega) \to \pi_1(\omega) \to \dots \to \pi_{k+1}(\omega) = y_n,
\end{equation*}
where $k = \# \Ln(\omega)$, then we may build a path $\pi(\eta)$ from $x$ to $y$ through another realization $X_1(\eta)$,\dots,$X_n(\eta)$ as follows.  For each $i \in \braces{ 1,\dots,k }$, let $j$ denote the index where $X_j(\omega) = \pi_i(\omega)$.  If $X_j(\omega) = X_j(\eta)$, then set $\pi_i(\eta) = \pi_i(\omega)$.  Otherwise, since $\eta \in H_n \cap W_n(a) \subset H_n$, there exists some $l$ such that $X_l(\eta)$ is in $B(z; R_2)$ and $\abs{X_l(\eta) - \pi_i(\omega)} < \zeta_n$.
Set $\pi_i(\eta) = X_l(\eta)$.
Then it follows that $\abs{\pi_i(\eta) - \pi_i(\omega)} \leq \zeta_n$ for all $i = 1,\dots,k$.
At the same time, $\abs{\pi_i(\omega) - \pi_{i+1}(\omega)} \leq \zeta_n$ for all $i = 1,\dots,k$ since $\omega \in F_n$.
It follows from the triangle inequality that
\begin{equation*}
\begin{split}
  \MoveEqLeft\abs{\pi_i(\eta) - \pi_{i+1}(\eta)}
  \\ &\leq \abs{\pi_i(\eta) - \pi_i(\omega)} + \abs{\pi_i(\omega) - \pi_{i+1}(\omega)} + \abs{\pi_{i+1}(\omega) - \pi_{i+1}(\eta)}
  \\ &\leq 3 \zeta_n.
\end{split}
\end{equation*}
Let $I$ be the set of indices $i$ where $\pi_i(\omega) \neq \pi_i(\eta)$.
Then the power-weighted length of the path $\pi(\eta)$ is at most
\begin{equation*}
\begin{split}
  \sum_{i=0}^{k} \abs{\pi_i(\eta) - \pi_{i+1}(\eta)}^p
  &\leq \sum_{i=0}^{k} \abs{\pi_i(\omega) - \pi_{i+1}(\omega)}^p + \sum_{\mathclap{\substack{i \in I\\\text{ or } i+1 \in I}}} \abs{\pi_i(\eta) - \pi_{i+1}(\eta)}^p
  \\ &\leq \Ln(\omega) + 2 \abs{I} (3 \zeta_n)^p.
\end{split}
\end{equation*}
On the other hand, $\eta \in W_n(a)$, i.e., $\Ln(\eta) \geq a$.
Hence
\begin{equation*}
  a \leq \Ln(\omega) + 2 \abs{I} (3 \zeta_n)^p.
\end{equation*}

Let $d_c(\omega; H_n \cap W_n(a))$ be the convex distance as defined in \eqref{eq:convex-distance-def}.
By choosing $s_i$ in \eqref{eq:convex-distance-def} as
\begin{equation*}
  s_i =
  \begin{cases}
    1 / \sqrt{\#\Ln(\omega)}, & \text{if $X_i$ is in the path $\Ln(\omega)$},\\
    0 & \text{otherwise},
  \end{cases}
\end{equation*}
there exists $\eta \in H_n \cap W_n(a)$ such that $\abs{I} \leq d_c(\omega; H_n \cap W_n(a)) \sqrt{\# \Ln(\omega)}$.
Therefore
\begin{equation*}
    \Ln(\omega)
    \geq a - 2 \cdot (3 \zeta_n)^p \cdot d_c(\omega; H_n \cap W_n(a)) \sqrt{ \#\Ln(\omega) }.
\end{equation*}
In particular, if $\Ln(\omega) \leq a - u$ for $u > 0$, then
\begin{align}
\notag
    d_c(\omega; H_n \cap W_n(a))
    &\geq \frac{u}{2 \cdot (3 \zeta_n)^p \cdot \sqrt{ \#\Ln(\omega) }} \\
\label{eq:path-length-bound}
    &\geq \frac{u}{2 \cdot (3 \zeta_n)^p \cdot \sqrt{C_* \tau_n} }
\end{align}
since $\omega \in G_n$.
Let $t > 0$ equal to the right side of \eqref{eq:path-length-bound}, and let $(H_n \cap W_n(a))_t$ denote the \emph{enlargement} of $H_n \cap W_n(a)$ as defined in \eqref{eq:talagrand-enlargement}, i.e., the collection of all elementary outcomes whose convex distance from $H_n \cap W_n(a)$ is at most $t$.
Then \eqref{eq:path-length-bound} implies that
\begin{equation}\label{eq:talagrand-pre}
    \Pr \paren[1]{ \braces{ \Ln \leq a - u } \cap F_n \cap G_n }
    \leq 1 - \Pr\paren[1]{(H_n \cap W_n(a))_t}.
\end{equation}

Let $M_n$ be the median of \Ln.
Note that $\Pr(H_n \cap W_n(M_n))$ is arbitrarily close to $1/2$ for sufficiently large $n$ since $\Pr(W_n(M_n)) = \Pr \paren{\Ln \geq M_n} = 1/2$, and $\Pr (H_n)$ approaches one as $n\to\infty$ by \autoref{thm:euc-shortlink}.
In particular, for $n$ sufficiently large, $\Pr(H_n \cap W_n(M_n)) \geq 1/3$.
Set $a = M_n$ in \eqref{eq:talagrand-pre}, and apply Talagrand's inequality \eqref{eq:talagrand-ineq} with $A = H_n \cap W_n(M_n)$ to obtain
\begin{equation*}
    \Pr \{ \Ln \leq M_n - u \}
    \leq 3 \exp\paren*{ -\frac{C_2 u^2}{\zeta_n^{2p} \tau_n}}
         + \paren[1]{1 - \Pr(F_n)} + \paren[1]{1 - \Pr(G_n)},
\end{equation*}
for sufficiently large $n$, where $C_2 = (2^4 3^{2p} C_*)^{-1}$.

To obtain an upper bound, set $a = M_n + u$.  From \eqref{eq:talagrand-pre}, $1 - \Pr\paren[1]{(H_n \cap W_n(a))_t} \geq 1/3$ for sufficiently large $n$ since both $\Pr (F_n)$ and $\Pr (G_n)$ converge to one as $n\to\infty$.  Apply Talagrand's inequality again for $A = H_n \cap W_n(a) = H_n \cap \braces{\Ln \geq M_n + u}$.
This gives
\begin{equation*}
    \Pr \{ \Ln \geq M_n + u \}
    \leq 3 \exp\paren*{ -\frac{C_2 u^2}{\zeta_n^{2p} \tau_n} } + \paren[1]{1 - \Pr(H_n)}
\end{equation*}
for sufficiently large $n$.
Combine the above two inequalities above and $\abs{x-y} \geq b$ to obtain
\begin{equation}\label{eq:Ln-bound}
  \begin{aligned}
    \Pr \paren*{ \frac{\abs{\Ln - M_n}}{ (n f(z))^{(1-p)/d} \abs{x-y} } > u }
      & \leq 6 \exp\paren[2]{ -C_2 (n f(z) b^d)^\alpha u^2 } + h_n
      \\ & = 6 \exp\paren[2]{ -C_3 n^\alpha u^2 } + h_n,
  \end{aligned}
\end{equation}
where $h_n = \paren[1]{1 - \Pr (F_n)} + \paren[1]{1 - \Pr (G_n)} + \paren[1]{1 - \Pr (H_n)}$, and $C_3 = C_2 (f(z) b^d)^\alpha$.
The reader can verify the inequality by recalling the definitions $\tau_n = b (nf(z))^{1/d}$, $\zeta_n = \tau_n^\alpha (nf(z))^{-1/d}$, and $\alpha = 1/(d+2p)$.

Note that $\Pr \paren{\abs{\Ln - M_n} > u} = 0$ when $u \geq \abs{x-y}^p$.
Integrate the right side of \eqref{eq:Ln-bound} for $u \geq 0$ to obtain the upper bound
\begin{align*}
    \frac{\abs{\BbbE \Ln - M_n}}{ (n f(z))^{(1-p)/d} \abs{x-y} }
    &\leq 6 \sqrt{ \frac{\pi}{C_3 n^\alpha} }
        + \paren[2]{(n f(z))^{1/d} \abs{x-y}}^{p-1} h_n.
\end{align*}
Since $\Pr(F_n)$, $\Pr(G_n)$, and $\Pr(H_n)$ approach one exponentially fast in $n^\alpha$, so does $h_n$.
Furthermore, the convergence rate is independent of the choice $x,y$.
Therefore
\begin{equation*}
    \lim_{n\to\infty} \frac{\abs{\BbbE \Ln - M_n}}{(n f(z))^{(1-p)/d} \abs{x-y}} = 0.
\end{equation*}
By \autoref{thm:Ln-mean}, for all sufficiently large $n$ and $\abs{x-y} \geq b$,
\begin{equation*}
    \Pr \paren[3]{ \abs[3]{ \frac{\Ln}{(n f(z))^{(1-p)/d} \abs{x-y}} - C(d,p) } > \vareps }
    \leq \Pr \paren[3]{ \frac{\abs{\Ln - M_n}}{ (n f(z))^{(1-p)/d} \abs{x-y}} > \frac{\vareps}{2} }.
\end{equation*}
Thus the proposition follows from \eqref{eq:Ln-bound}, \autoref{thm:euc-shortlink}, and \autoref{thm:Ln-size}.
\end{proof}

\begin{theorem}\label{thm:euc-unif-local}
Let $z \in \reals^d$, $R > 0$, and $b \in (0, R/2)$.
Assume that the pdf $f$ is uniform in $B(z; R)$ but may have probability mass outside.
Choose $\vareps > 0$ sufficiently small so that $(C(d,p) + \vareps) / 2 < (C(d,p) - \vareps) 5 / 8$.

Denote by $E_n(\vareps)$ the event that for all $x \in B(z; R/4)$ and $u \notin B(z; R)$,
\begin{equation}
  \label{eq:uniform-exit}
  \frac{L(x,u; \cX_n)}{(n f(z))^{(1-p)/d}} > (C(d,p) - \vareps) \cdot \frac58 R.
\end{equation}
Denote by $E'_n(\vareps)$ the event that for all $x,y \in B(z; R/4)$ with $\abs{x-y} \geq b$,
\begin{gather}
  \label{eq:Ln-equal-to-restriction}
  L(x,y; \cX_n) = L(x,y; \cX_n \cap B(z; R)),
\shortintertext{and}
  \label{eq:euc-unif-local-e1}
  \abs*{\frac{L(x,y; \cX_n)}{(n f(z))^{(1-p)/d} \abs{x-y}} - C(d,p)} \leq \vareps.
\end{gather}
Then there exists $\theta_4 > 0$ such that
\begin{equation*}
  1 - \Pr(E_n(\vareps) \cap E'_n(\vareps))
  \leq \exp\paren[1]{-\theta_4 n^\alpha}
\end{equation*}
for all sufficiently large $n$.
\end{theorem}


\autoref{thm:euc-unif-local} asserts that with high probability, for sufficiently large $n$ the shortest path between the points $x,y$ in the open ball $B(z; R/4)$ does not exit $B(z; R)$.
We first prove a lemma that will be used to prove \autoref{thm:euc-unif-local}.

\begin{lemma}\label{thm:Ln-continuous}
Suppose the assumptions in \autoref{thm:euc-unif-local} hold.
Let $u,v,x,y \in B(z; R)$.
If
\begin{enumerate}[label=C.\arabic*,align=left]
  \item\label{it:local-fn}: the event $F_n$ from \autoref{thm:euc-shortlink} occurs,
  \item\label{it:local-c1}: $\abs{x-u}$ and $\abs{y-v}$ are at most $b^\alpha (n f(z))^{(\alpha-1)/d}$,
  \item\label{it:local-c2}: $\abs{x-y} \geq b$, and
  \item\label{it:local-c3}: $\displaystyle \frac{L(u,v; \cX_n \cap B(z; R))}{(n f(z))^{(1-p)/d} \abs{u-v}} \leq C(d,p) + \vareps$,
\end{enumerate}
then there exists $n_0 > 0$ independent of the choice $u,v,x,y$ in $B(z; R)$ such that
\begin{equation}
  \label{eq:euc-local-approx}
  \abs*{ \frac{L(x,y; \cX_n \cap B(z; R))}{(n f(z))^{(1-p)/d} \abs{x-y}} - \frac{L(u,v; \cX_n \cap B(z; R))}{(n f(z))^{(1-p)/d} \abs{u-v}} } < \frac{\vareps}{2}
\end{equation}
for all $n \geq n_0$.
\end{lemma}

\begin{proof}
We bound the left side of \eqref{eq:euc-local-approx} from above by
\begin{equation}\label{eq:euc-local-difference}
\begin{split}
  \MoveEqLeft \abs*{ \frac{L(x,y; \cX_n \cap B(z; R))}{(n f(z))^{(1-p)/d} \abs{x-y}} - \frac{L(u,v; \cX_n \cap B(z; R))}{(n f(z))^{(1-p)/d} \abs{x-y}} }
    \\ &+ \abs*{ \frac{L(u,v; \cX_n \cap B(z; R))}{(n f(z))^{(1-p)/d} \abs{x-y}} - \frac{L(u,v; \cX_n \cap B(z; R))}{(n f(z))^{(1-p)/d} \abs{u-v}} }.
\end{split}
\end{equation}
Since the event $F_n$ occurred by \autoref{it:local-fn}, every edge in the shortest path has length at most $b^\alpha (n f(z))^{(\alpha-1)/d}$.
Combining with \autoref{it:local-c1}, the difference between $L(x,y; \cX_n \cap B(z; R))$ and $L(u,v; \cX_n \cap B(z; R))$ is at most $2 \cdot (2b^\alpha)^p (n f(z))^{(\alpha-1)p/d}$.
With \autoref{it:local-c2}, the first term in \eqref{eq:euc-local-difference} may be bounded from above as follows:
\begin{equation*}
    \abs*{ \frac{L(x,y; \cX_n \cap B(z; R))}{(n f(z))^{(1-p)/d} \abs{x-y}} - \frac{L(u,v; \cX_n \cap B(z; R))}{(n f(z))^{(1-p)/d} \abs{x-y}} }
    \leq \frac{ 2^{p+1} b^{\alpha p} (n f(z))^{(\alpha-1)p/d} }{b (n f(z))^{(1-p)/d}}.
\end{equation*}
Note that $(\alpha-1)p/d - (1-p)/d = -(1 - \alpha p) / d < 0$.  Therefore, there exists $n_1$ such that the first term in \eqref{eq:euc-local-difference} is smaller than $\vareps/4$ for all $n \geq n_1$, and $n_1$ is independent of the choice $u,v,x,y$.

Since $\abs[1]{\abs{u-v} - \abs{x-y}} < 2 b^\alpha (nf(z))^{(\alpha-1)/d}$ by \autoref{it:local-c1}, the second term in \eqref{eq:euc-local-difference} can be bounded from above as follows using \autoref{it:local-c3}:
\begin{align*}
    \MoveEqLeft \abs*{ \frac{L(u,v; \cX_n \cap B(z; R))}{(n f(z))^{(1-p)/d} \abs{x-y}} - \frac{L(u,v; \cX_n \cap B(z; R))}{(n f(z))^{(1-p)/d} \abs{u-v}} }
    \\ &\leq \frac{L(u,v; \cX_n \cap B(z; R))}{(n f(z))^{(1-p)/d} \abs{u-v}} \abs*{ \frac{\abs{u-v} - \abs{x-y}}{\abs{x-y}} }
    \\ &\leq \paren[2]{C(d,p) + \vareps} \frac{2 b^\alpha (n f(z))^{(\alpha-1)/d}}{b}.
\end{align*}
Since $(\alpha-1)/d < 0$, there exists $n_2$ such that the second term in \eqref{eq:euc-local-difference} is smaller than $\vareps/4$ for all $n \geq n_2$, and again $n_2$ is independent of the choice $u,v,x,y$.
\autoref{thm:Ln-continuous} follows by choosing $n_0 = \max(n_1,n_2)$.
\end{proof}

\begin{proof}[Proof of \autoref{thm:euc-unif-local}]
Let $\zeta_n = b^\alpha (nf(z))^{(\alpha-1)/d}$.
For a set of points $\braces{w_i}_{i=1}^m$ in $\reals^d$, let
\begin{equation*}
  \braces*{ B(w_i; \zeta_n) \colon w_i \in B(z; R/4), 1 \leq i \leq m}
\end{equation*}
be a finite open cover of $B(z; R/4)$ with $m = O(n)$.
Likewise, for a set of points $\braces{v_k}_{k=1}^\ell$ in $\reals^d$, let
\begin{equation*}
  \braces*{B(v_k; \zeta_n) \colon v_k \in B(z; 7R/8), 1 \leq k \leq \ell}
\end{equation*}
be a finite open cover of the boundary of $B(z; 7R/8)$ with $\ell = O(n)$.

Suppose that
\begin{enumerate}[label=D.\arabic*,align=left]
  \item\label{it:local-d1}: the event $F_n$ from \autoref{thm:euc-shortlink} occurs, and
  \item\label{it:local-d2}: $\displaystyle \abs*{ \frac{L(w_i, v_k; \cX_n \cap B(z;R))}{(n f(z))^{(1-p)/d} \abs{w_i-v_k}} - C(d,p) } \leq \frac{\vareps}{2}$ for all $w_i,v_k$.
\end{enumerate}
We claim that \eqref{eq:uniform-exit} holds for sufficiently large $n$, under the assumptions \autoref{it:local-d1} and \autoref{it:local-d2}.

Let $L(x; r; \cX_n)$, $r > 0$, denote the minimal power-weighted path length from $x$ to the boundary of $B(z; r)$, i.e.,
\begin{equation*}
  L(x; r; \cX_n) = \min_{\abs{z-y} = r} L(x,y; \cX_n).
\end{equation*}
Every path from $x \in B(z; R/4)$ to $u \notin B(z; R)$ crosses the boundary of $B(z; 7R/8)$, therefore $L(x,u; \cX_n) \geq L(x; 7R/8; \cX_n)$.
It suffices to show that
\begin{equation}
  \label{eq:Lr-uniform-exit}
  \frac{L(x; 7R/8; \cX_n)}{(nf(z))^{(1-p)/d}} \geq (C(d,p) - \vareps) \frac58 R
\end{equation}
for all $x \in B(z; R/4)$ to prove \eqref{eq:uniform-exit}.

Note that $L(x; 7R/8; \cX_n) = L(x; 7R/8; \cX_n \cap B(z; R))$, i.e.,
\begin{equation*}
  L(x; 7R/8; \cX_n) = \min_{\abs{z-y} = 7R/8} L(x,y; \cX_n \cap B(z; R)).
\end{equation*}
If the shortest path to the boundary were to reach any point outside $B(z; R) \supset B(z; 7R/8)$, the path must have already passed through the boundary, which is a contradiction.

For every $x \in B(z; R/4)$, there exists $w_i$ such that $\abs{x - w_i} < \zeta_n$, and for every $q$ on the boundary of $B(z; 7R/8)$, there exists $v_k$ such that $\abs{q - v_k} < \zeta_n$.
Consequently, by \autoref{thm:Ln-continuous} and assumptions \autoref{it:local-d1} and \autoref{it:local-d2}, for sufficiently large $n$,
\begin{equation*}
  \frac{L(x,q; \cX_n \cap B(z;R))}{(n f(z))^{(1-p)/d} \abs{x-q}} \geq C(d,p) - \vareps,
\end{equation*}
for all $x \in B(z; R/4)$ and for all $q$ satisfying $\abs{z-q} = 7R/8$.
Use $\abs{x-q} \geq 5R/8$, and we have proved \eqref{eq:Lr-uniform-exit}, and in turn, \eqref{eq:uniform-exit}.

Now let $x,y \in B(z; R/4)$ and $\abs{x-y} \geq b$.
Then there exist $w_i,w_j$ such that $\abs{w_i - x} < \zeta_n$ and $\abs{w_j - y} < \zeta_n$.
Suppose the following condition holds in addition to \autoref{it:local-d1} and \autoref{it:local-d2};
\begin{enumerate}[resume,label=D.\arabic*,align=left]
  \item\label{it:local-d3}: $\displaystyle \abs*{\frac{L(w_i,w_j; \cX_n \cap B(z; R))}{(n f(z))^{(1-p)/d} \abs{w_i-w_j}} - C(d,p)} \leq \frac\vareps2$.
\end{enumerate}
We claim that \eqref{eq:Ln-equal-to-restriction} holds.

Assume to the contrary that the path $L(x,y; \cX_n)$ includes some point outside $B(z; R)$.
Then the path has crossed the boundary of $B(z; 7R/8)$, hence $L(x; 7R/8; \cX_n) \leq L(x,y; \cX_n)$.
We have already seen that \eqref{eq:Lr-uniform-exit} holds.
Therefore,
\begin{equation*}
  \paren[1]{C(d,p) - \vareps} \frac58 R
  \leq \frac{L(x; 7R/8; \cX_n)}{(n f(z))^{(1-p)/d}}
  \leq \frac{L(x, y; \cX_n)}{(n f(z))^{(1-p)/d}}.
\end{equation*}
On the other hand, apply \autoref{thm:Ln-continuous} with \autoref{it:local-d3} to have
\begin{equation*}
  \frac{L(x, y; \cX_n)}{(n f(z))^{(1-p)/d}}
  \leq \frac{L(x, y; \cX_n \cap B(z; R))}{(n f(z))^{(1-p)/d}}
  \leq \paren[1]{C(d,p) + \vareps} \frac12 R
\end{equation*}
since $x,y \in B(z; R/4)$ and $\abs{x-y} \leq R/2$.
Recall that $\vareps$ was assumed to be sufficiently small so that $(C(d,p) + \vareps) / 2 < (C(d,p) - \vareps) 5/8$.
Therefore we have a contradiction.

We now claim that \eqref{eq:euc-unif-local-e1} is true when assumptions \autoref{it:local-d1}, \autoref{it:local-d2}, and \autoref{it:local-d3} hold.
Start with
\begin{equation}\label{eq:euc-unif-local-3e}
\begin{split}
    \MoveEqLeft \abs*{ \frac{L(x,y; \cX_n)}{(n f(z))^{(1-p)/d} \abs{x-y}} - C(d,p) }
    \\ \leq {} &\abs*{ \frac{L(x,y; \cX_n)}{(n f(z))^{(1-p)/d} \abs{x-y}} - \frac{L(w_i,w_j; \cX_n)}{(n f(z))^{(1-p)/d} \abs{w_i-w_j}}}
    \\ & + \abs*{ \frac{L(w_i,w_j; \cX_n)}{(n f(z))^{(1-p)/d} \abs{w_i-w_j}} - C(d,p) }.
\end{split}
\end{equation}
From \eqref{eq:uniform-exit} and \eqref{eq:Ln-equal-to-restriction}, $L(x,y; \cX_n)$ and $L(w_i,w_j; \cX_n)$ in the upper bound of \eqref{eq:euc-unif-local-3e} can be replaced by $L(x,y; \cX_n \cap B(z; R))$ and $L(w_i,w_j; \cX_n \cap B(z; R))$, respectively.
Therefore, for sufficiently large $n$, the first term in the upper bound is less than $\vareps / 2$ by \autoref{thm:Ln-continuous}, and the second term is less than $\vareps / 2$ by \autoref{it:local-d3}.
This establishes that \eqref{eq:euc-unif-local-e1} holds.


In summary, we have shown that \eqref{eq:uniform-exit}, \eqref{eq:Ln-equal-to-restriction}, and \eqref{eq:euc-unif-local-e1} hold when events \autoref{it:local-d1}, \autoref{it:local-d2}, and \autoref{it:local-d3} occur.
If the event $E_n(\vareps)$ or $E'_n(\vareps)$ does not occur, either one of \autoref{it:local-d1}, \autoref{it:local-d2}, or \autoref{it:local-d3} does not occur:
\begin{multline*}
    1 - \Pr\paren[1]{E_n(\vareps) \cap E'_n(\vareps)}
    \leq \paren[1]{1 - \Pr\paren[1]{F_n}}
    \\
    \begin{aligned}
      + \sum_{w_i,v_k} \Pr \paren[3]{ \abs[3]{ \frac{L(w_i,v_k; \cX_n \cap B(z;R))}{(n f(z))^{(1-p)/d} \abs{w_i-v_k}} - C(d,p) } > \frac{\vareps}{2} }&
      \\ + \sum_{w_i,w_j} \Pr \paren[3]{ \abs[3]{ \frac{L(w_i,w_j; \cX_n \cap B(z;R))}{(n f(z))^{(1-p)/d} \abs{w_i-w_j}} - C(d,p) } > \frac{\vareps}{2} }&.
    \end{aligned}
\end{multline*}
The first sum is over all pairs $w_i,v_k$.
The second sum is over all $w_i,w_j$ with $\abs{w_i-w_j} \geq b/2$.
\autoref{thm:euc-unif-local} now follows from \autoref{thm:euc-shortlink} and \autoref{thm:euc-unif} with $R_1 = 7R/8$ and $R_2 = R$, since the number of summands are of polynomial order in $n$.
\end{proof}

\begin{corollary}\label{thm:euc-local}
Assume that $f$ is continuous at $z \in \reals^d$ and $f(z) > 0$.
Fix $\vareps > 0$ small enough so that $(C(d,p) + \vareps) / 2 < (C(d,p) - \vareps) 5/8$.
For $R > 0$ and $b \in (0,R/2)$, let $H_n = H_n(z,R,\vareps,b)$ denote the event that
\begin{equation}\label{eq:euc-local}
  \abs*{\frac{L(x,y; \cX_n)}{(n f(z))^{(1-p)/d} \abs{x-y}} - C(d,p)} \leq \vareps,
\end{equation}
for all $x,y \in B(z; R/4)$ with $\abs{x-y} \geq b$, and simultaneously
\begin{equation}\label{eq:euc-exit}
  \frac{L(x,u; \cX_n)}{(n f(z))^{(1-p)/d}}
  \geq (C(d,p) - \vareps) \frac58 R,
\end{equation}
for all $x \in B(z; R/4)$ and $u \notin B(z; R)$.
Then there exists $R = R(z) > 0$ such that for all $b \in (0,R/2)$ there exists $\theta_5 > 0$ for which,
\begin{equation}\label{eq:euc-local-tails}
  1 - \Pr(H_n(z,R,\vareps,b))
  \leq \exp\paren[1]{-\theta_5 n^\alpha}
\end{equation}
for all sufficiently large $n$.
\end{corollary}

\begin{proof}
Let $f_m = f_m(R)$ and $f_M = f_M(R)$ denote the infimum and the supremum of $f$ inside $B(z; R)$, respectively.
Since $f$ is continuous at $z$, we can choose sufficiently small $R = R(z) > 0$ such that $f_m$ and $f_M$ are sufficiently close to $f(z)$ to satisfy
\begin{gather}
  \label{eq:unif-local-r1}
  \paren[1]{C(d,p) + \vareps/2} \paren*{\frac{f(z)}{f_m}}^{(p-1)/d} \leq C(d,p) + \vareps,\text{ and}
  \\ \label{eq:unif-local-r2}
  \paren[1]{C(d,p) - \vareps/2} \paren*{\frac{f(z)}{f_M}}^{(p-1)/d} \geq C(d,p) - \vareps.
\end{gather}
In addition, shrink $R$ if necessary to ensure $f_M V_d R^d \leq 1$ where $V_d$ denotes the volume of a unit ball.
Now choose and fix an arbitrary $b \in (0,R/2)$ for the entire proof.

For each sample point $X_i \in \cX_n$,
let $Y_i$ be an arbitrary random point outside $B(z; R)$, and
let $Z_i$ be an independent Bernoulli random variable with $\Pr(Z_i = 1) = f_m / f(X_i)$.
Define a new random variable
\begin{equation*}
  X_i^m =
  \begin{cases}
    X_i & \text{if $X_i$ is not in $B(z; R)$},\\
    X_i Z_i + Y_i (1 - Z_i) & \text{if $X_i$ is in $B(z; R)$}.
  \end{cases}
\end{equation*}
Let $\cX_n^m = \braces{ X_1^m, \dots, X_n^m }$.
$\cX_n^m$ is an i.i.d.\ sample and its pdf restricted to $B(z; R)$ is uniform with intensity $f_m$.
Define
\begin{enumerate}[label=E.\arabic*,align=left]
  \item\label{it:local-e1}: both the events $E_n(\vareps/2)$ and $E'_n(\vareps/2)$ in \autoref{thm:euc-unif-local} occur for $\cX_n^m$.
\end{enumerate}
Assume that the event \autoref{it:local-e1} occurs.
Let $x,y \in B(z; R/4)$ with $\abs{x-y} \geq b$.
Since $(\cX_n^m \cap B(z; R)) \subset (\cX_n \cap B(z; R))$, we have
\begin{equation*}
  L(x,y; \cX_n)
  \leq L(x,y; \cX_n \cap B(z; R))
  \leq L(x,y; \cX_n^m \cap B(z; R))
  = L(x,y; \cX_n^m).
\end{equation*}
The last equality comes from \eqref{eq:Ln-equal-to-restriction} for $\cX_n^m$.
From \eqref{eq:euc-unif-local-e1} and \eqref{eq:unif-local-r1}, we have
\begin{equation}
\label{eq:euc-local-upper}
  \frac{L(x,y; \cX_n)}{(nf(z))^{(1-p)/d} \abs{x-y}}
  \leq \frac{L(x,y; \cX_n^m)}{(nf_m)^{(1-p)/d} \abs{x-y}} \cdot \paren*{\frac{f(z)}{f_m}}^{\frac{p-1}{d}}
  \leq C(d,p) + \vareps.
\end{equation}
This establishes the upper half of the inequality \eqref{eq:euc-local} under \autoref{it:local-e1}.

It remains to establish the lower half of the inequality \eqref{eq:euc-local}.
This is established in two steps.
First we show that \eqref{eq:euc-exit} holds assuming an event \autoref{it:local-e2} analogous to \autoref{it:local-e1}.
Then we show that \autoref{it:local-e1} and \autoref{it:local-e2} imply \eqref{eq:euc-local}.

For each point $X_i \in \cX_n$, define a new random variable $X_i^M$ as follows.
Let $\sigma = \int (f_M - f(u)) \,du \geq 0$ where the integral is taken inside $B(z; R)$.
By the assumption $f_M V_d R^d \leq 1$, we have $0 \leq \sigma \leq 1$.
Let $\tilde Y_i$ be a random point inside $B(z; R)$ with pdf $\sigma^{-1} (f_M - f(u))$ for $u \in B(z; R)$, and
let $\tilde Z_i$ be a Bernoulli random variable with $\Pr(\tilde Z_i = 1) = 1 - \sigma$.
Define
\begin{equation*}
  X_i^M =
  \begin{cases}
    X_i & \text{if $X_i$ is in $B(z; R)$},\\
    X_i \tilde Z_i + \tilde Y_i (1 - \tilde Z_i) & \text{if $X_i$ is not in $B(z; R)$}.
  \end{cases}
\end{equation*}
Let $\cX_n^M = \braces{ X_1^M, \dots, X_n^M }$.
$\cX_n^M$ is an i.i.d.\ sample and its pdf restricted to $B(z; R)$ is uniform with intensity $f_M$.
Define
\begin{enumerate}[resume,label=E.\arabic*,align=left]
  \item\label{it:local-e2}: both the events $E_n(\vareps/2)$ and $E'_n(\vareps/2)$ in \autoref{thm:euc-unif-local} occur for $\cX_n^M$.
\end{enumerate}

Assume that the event \autoref{it:local-e2} occurs.
Let $x \in B(z; R/4)$ and $v = \argmin_y L(x,y; \cX_n)$ over all $\abs{z-y} = R$.
Then $L(x,v; \cX_n) = L(x,v; \cX_n \cap B(z; R))$, otherwise the shortest path from $x$ to $v$ has passed through another point on the boundary of $B(z; R)$, and this contradicts the choice of $v$.
Since $(\cX_n \cap B(z; R)) \subset (\cX_n^M \cap B(z; R))$,
\begin{equation*}
  L(x,v; \cX_n)
  = L(x,v; \cX_n \cap B(z; R))
  \geq L(x,v; \cX_n^M \cap B(z; R))
  \geq L(x,v; \cX_n^M).
\end{equation*}
As $v \notin B(z; R)$, it follows from \eqref{eq:unif-local-r2} and \eqref{eq:uniform-exit} for $\cX_n^M$ that
\begin{equation*}
  \frac{L(x,v; \cX_n)}{(nf(z))^{(1-p)/d}}
  \geq \frac{L(x,v; \cX_n^M)}{(nf_M)^{(1-p)/d}} \cdot \paren*{\frac{f(z)}{f_M}}^{\frac{p-1}{d}}
  \geq (C(d,p) - \vareps) \frac58 R.
\end{equation*}
If $u \notin B(z; R)$, the path $L(x,u; \cX_n)$ crosses the boundary of $B(z; R)$ at some point $u^\prime$, hence $L(x,u; \cX_n) \geq L(x,u^\prime; \cX_n) \geq L(x,v; \cX_n)$ by the minimality of $v$.
This establishes \eqref{eq:euc-exit} under \autoref{it:local-e2}.

Now we show that the upper bound of \eqref{eq:euc-local} holds under the conditions \autoref{it:local-e1} and \autoref{it:local-e2}.
Let $x,y \in B(z; R/4)$.
Then $L(x,y; \cX_n) = L(x,y; \cX_n \cap B(z; R))$.
Otherwise, i.e., if $L(x,y; \cX_n)$ reaches some point $u \notin B(z; R)$, then $L(x,u; \cX_n) \leq L(x,y; \cX_n)$ but this contradicts \eqref{eq:euc-exit} and \eqref{eq:euc-local-upper} since $\abs{x-y} \leq R/2$ and $\vareps$ was assumed to satisfy $(C(d,p) + \vareps) / 2 < (C(d,p) - \vareps) 5/8$.
We can now repeat the same argument used to establish that the upper half of the inequality \eqref{eq:euc-local} follows from \autoref{it:local-e1} to show that
\begin{equation*}
  \frac{L(x,y; \cX_n)}{(nf(z))^{(1-p)/d} \abs{x-y}}
  \geq \frac{L(x,y; \cX_n^M)}{(nf_M)^{(1-p)/d} \abs{x-y}} \cdot \paren*{\frac{f(z)}{f_M}}^{\frac{p-1}{d}}
  \geq C(d,p) - \vareps
\end{equation*}
by \eqref{eq:Ln-equal-to-restriction} and \eqref{eq:euc-unif-local-e1} for $\cX_n^M$, and \eqref{eq:unif-local-r2}.

Applying \autoref{thm:euc-unif-local} once to $\cX_n^m$ and once to $\cX_n^M$ asserts that \autoref{it:local-e1} and \autoref{it:local-e2} occur with exponentially small probability, which establishes \autoref{thm:euc-local}.
\end{proof}

\subsection{Convergence in Riemannian manifolds}\label{sec:manifolds}
We adapt \autoref{thm:euc-local} to the case when the probability distribution is supported on a Riemannian manifold $M$ instead of on a Euclidean space.
For $z \in M$ and $R > 0$, $B(z; R)$ denotes the set $\braces{ u \in M \colon \dist_1(u,z) < R}$.
Recall that $\alpha = 1/(d+2p)$.

\begin{lemma}\label{thm:Ln-normal-chart}
Let $(M,g_1)$ be a Riemannian manifold equipped with metric tensor $g_1$.
Let $z \in M$ be a point, and let $\vareps > 0$ be a fixed constant.
For $R > 0$ and $b \in (0,2R)$ we denote by $E_n(B(z; R), \vareps, b)$ the event that
\begin{enumerate}[label=(\roman*)]
  \item \label{it:normal-chart-e1} if a shortest path passes through $B(z; R)$ then its edges in $B(z; R)$ have $\dist_1$-lengths at most $b^\alpha (n f(z))^{(\alpha-1)/d}$, and
  \item \label{it:normal-chart-e2} for every pair $x,y \in B(z; R)$ with $\dist_1(x,y) \geq b$,
\begin{equation}\label{eq:normal-chart-e2-eq}
    \abs[3]{ \frac{L(x,y; \cX_n)}{n^{(1-p)/d} \dist_p(x,y)} - C(d,p) } \leq \vareps.
\end{equation}
\end{enumerate}

Assume that $f(z) > 0$ and that $f$ is continuous at $z$.
Then there exists $R = R(z) > 0$ such that for all $b \in (0,2R)$ there exists $\theta_6 > 0$ for which,
\begin{equation}\label{eq:normal-chart-expdecay}
  1 - \Pr(E_n(B(z; R), \vareps, b))
  \leq \exp\paren[1]{-\theta_6 n^\alpha}
\end{equation}
for all sufficiently large $n$.
\end{lemma}

In \autoref{it:normal-chart-e1}, a shortest path edge between two sample points is contained in $B(z; R)$ when a $g_1$-geodesic between these points is contained in $B(z; R)$.


\begin{proof}
We prove this lemma by showing that both the events \autoref{it:normal-chart-e1} and \autoref{it:normal-chart-e2} satisfy the probability tail bound as in \eqref{eq:normal-chart-expdecay}.
That the statement \eqref{eq:normal-chart-expdecay} holds for the event \autoref{it:normal-chart-e1} follows from similar arguments as used in \autoref{thm:euc-shortlink}.
We focus on the event \autoref{it:normal-chart-e2}.

Choose $\delta > 0$ small enough to ensure that
\begin{gather}
  \label{eq:normal-chart-delta}
  2 (1+\delta)^p \leq \frac52 (1-\delta)^p,
  \\
  \label{eq:normal-chart-delta-1}
  \paren[1]{(1+\delta) / (1-\delta)}^p \paren[1]{C(d,p) + \vareps/2} \leq C(d,p) + \vareps,\text{ and}
  \\
  \notag
  \label{eq:normal-chart-delta-2}
  \paren[1]{(1-\delta) / (1+\delta)}^p \paren[1]{C(d,p) - \vareps/2} \geq C(d,p) - \vareps.
\end{gather}
Define $U = B(z; 4R) = \braces{ u \in M \colon \dist_1(u,z) < 4R }$ for $R > 0$.  Since $f$ is continuous, we may choose $R > 0$ small enough so that there exists a normal chart map $\varphi \colon U \subset M \to V \subset \reals^d$ such that the event $H_n = H_n(\varphi(z), 4R, 2^{-1} \vareps, (1+\delta)^{-1} b)$ from \autoref{thm:euc-local} satisfies \eqref{eq:euc-local-tails}, and that
\begin{gather}
\label{eq:Ln-normal-chart-eq1}
    (1 - \delta)^d \sup\nolimits_U f
    \leq f(z) = f(\varphi(z))
    \leq (1 + \delta)^d \inf\nolimits_U f,
\\
\label{eq:Ln-normal-chart-eq2}
    1 - \delta
    \leq \frac{\dist_1(u,v)}{\abs{\varphi(u)-\varphi(v)}}
    \leq 1 + \delta,
\end{gather}
for all $u,v \in U$ with $u \neq v$.
Recall that $f(z) = f(\varphi(z))$ follows from the basic properties of normal coordinates at $z \in M$.
See, e.g., \textcite[p.~73]{Oneill:1983a}.
The denominator in \eqref{eq:Ln-normal-chart-eq2} is a Euclidean distance.

We claim that \autoref{it:normal-chart-e2} is true when the event $H_n(\varphi(z), 4R, 2^{-1} \vareps, (1+\delta)^{-1} b)$ occurs.
Then \eqref{eq:normal-chart-expdecay} would follow from \autoref{thm:euc-local}.
In the remainder of this proof, we prove this claim.

Let $x,y \in B(z; R) \subset U$.
Then a $g_1$-geodesic curve from $x$ to $y$ is contained in $U$ by the triangle inequality.
It follows from the definition of $\dist_p$ in \eqref{eq:distp_def} that
\begin{equation}\label{eq:Ln-normal-chart-eq3}
    \dist_p(x,y) \leq \dist_1(x,y) (\inf\nolimits_U f)^{(1-p)/d}.
\end{equation}
Furthermore, if a $g_p$-geodesic curve from $x$ to $y$ were contained in $U$, then
\begin{equation}\label{eq:Ln-normal-chart-eq4}
    \dist_p(x,y) \geq \dist_1(x,y) (\sup\nolimits_U f)^{(1-p)/d}.
\end{equation}
If a $g_p$-geodesic curve from $x$ to $y$ exits $U$, then $\dist_p(x,y)$ must be at least $(3 R) (\sup\nolimits_R f)^{(1-p)/d}$ by the assumptions $\dist_1(x,z) < R$ and $U = B(z; 4R)$.
On the other hand, it follows from \eqref{eq:normal-chart-delta}, \eqref{eq:Ln-normal-chart-eq1}, and \eqref{eq:Ln-normal-chart-eq3} that
\begin{align*}
    \dist_p(x,y)
    & \leq \dist_1(x,y) (\inf\nolimits_U f)^{(1-p)/d}
    \\ & \leq (2 R) (\sup\nolimits_U f)^{(1-p)/d} \paren*{\frac{1+\delta}{1-\delta}}^{p-1}
    \\ & < \frac52 R (\sup\nolimits_U f)^{(1-p)/d},
\end{align*}
and this is a contradiction.
Therefore a $g_p$-geodesic curve from $x$ to $y$ does not exit $U$, hence \eqref{eq:Ln-normal-chart-eq4} holds.

Next we show that $L(x,y; \cX_n) = L(x,y; \cX_n \cap U)$, i.e., the shortest path between $x,y \in B(z; R)$ is contained in $U$, when $H_n(\varphi(z), 4R, 2^{-1} \vareps, (1+\delta)^{-1} b)$ occurs.
Assume to the contrary that the path $L(x,y; \cX_n)$ from $x$ exits $U$.
Then the corresponding path in $V$ starts from $\varphi(x)$ and exits $V$, and its power-weighted length is at least $(C(d,p) - \vareps) 5R / 2$.
Note that in this proof \autoref{thm:euc-local} was applied with $R$ replaced by $4R$.
By \eqref{eq:Ln-normal-chart-eq2} and \eqref{eq:euc-exit}, it implies that $L(x,y; \cX_n) / (nf(z))^{(1-p)/d}$ is at least $(C(d,p) - \vareps) (1-\delta)^p 5R / 2$.
On the other hand, by \eqref{eq:Ln-normal-chart-eq2} and \eqref{eq:euc-local}, $L(x,y; \cX_n \cap U) / (nf(z))^{(1-p)/d}$ is at most $(C(d,p) + \vareps) (1+\delta)^p 2R$.
This is a contradiction by \eqref{eq:normal-chart-delta}, so we conclude $L(x,y; \cX_n) = L(x,y; \cX_n \cap U)$.

Let $L(\varphi(x),\varphi(y); \varphi(\cX_n \cap U))$ denote the shortest path length between $\varphi(x),\varphi(y) \in V$ in Euclidean space $\reals^d$.
Then
\begin{align*}
  \frac{L(x,y; \cX_n)}{n^{(1-p)/d} \dist_p(x,y)}
  &\leq \frac{L(x,y; \cX_n)}{(n \sup\nolimits_U f)^{(1-p)/d} \dist_1(x,y)}
  \\
  &\leq \paren*{\frac{1+\delta}{1-\delta}}^p \frac{L(\varphi(x), \varphi(y); \varphi(\cX_n \cap U))}{(n f(\varphi(z)))^{(1-p)/d} \abs{\varphi(x)-\varphi(y)}}
  \\
  &\leq C(d,p) + \vareps.
\end{align*}
The first inequality follows from \eqref{eq:Ln-normal-chart-eq4}, the second one follows from \eqref{eq:Ln-normal-chart-eq1} and \eqref{eq:Ln-normal-chart-eq2}, and the third one follows from \eqref{eq:normal-chart-delta-1} and the assumption that $H_n(\varphi(z),4R,2^{-1}\vareps,(1+\delta)^{-1}b)$ occurred.
Repeat the same argument for the lower bound to obtain
\begin{equation*}
  \frac{L(x,y; \cX_n)}{n^{(1-p)/d} \dist_p(x,y)}
  \geq C(d,p) - \vareps.
\end{equation*}
The last two inequalities imply \eqref{eq:normal-chart-e2-eq}.
\end{proof}

Our main result \autoref{thm:main} can now be obtained by applying \autoref{thm:Ln-normal-chart} to a finite open cover of the compact manifold $M$.

\begin{figure}
\centering
\begin{tikzpicture}
    \draw (0,0) .. controls (2,2) and (4,-2) .. (6,0)
        node[pos=0]  (z1) {}
        node[pos=.3] (z2) {}
        node[pos=.7] (z3) {}
        node[pos=1]  (z4) {};

    \path (z1) -- +(.2,.3)  node (u1) {};
    \path (z2) -- +(0,-.3)  node (u2) {};
    \path (z3) -- +(-.1,-.2) node (u3) {};

    \fill (z1) circle (2pt) node[below] {$z_1$};
    \draw (u1) circle (0.7) circle (2.1);
    \path (u1) -- +(120:0.7) node [above] {$U_1$};
    \path (u1) -- +(150:2.1) node [above left] {$V_1$};

    \fill (z2) circle (2pt) node[below left] {$z_2$};
    \draw (u2) circle (1.0) circle (3.0);
    \path (u2) -- +(-120:0.9) node [below left] {$U_2$};
    \path (u2) -- +(60:3.0) node [above right] {$V_2$};

    \fill (z3) circle (2pt) node[above] {$z_3$};
    \draw (u3) circle (0.8) circle (2.4);
    \path (u3) -- +(-30:0.8) node [right] {$U_3$};
    \path (u3) -- +(-30:2.4) node [below right] {$V_3$};

    \fill (z4) circle (2pt) node[above] {$z_4$};
\end{tikzpicture}
\caption{Path division procedure described in the proof of \autoref{thm:main}.  Here $k = 4$.  Note that $z_i \in U_i$ and $z_{i+1} \in V_i$ for $i = 1,2,3$.  Shortest path is depicted as a smooth curve for illustration purpose only and it is actually piecewise smooth.}
\label{fig:path-div}
\end{figure}
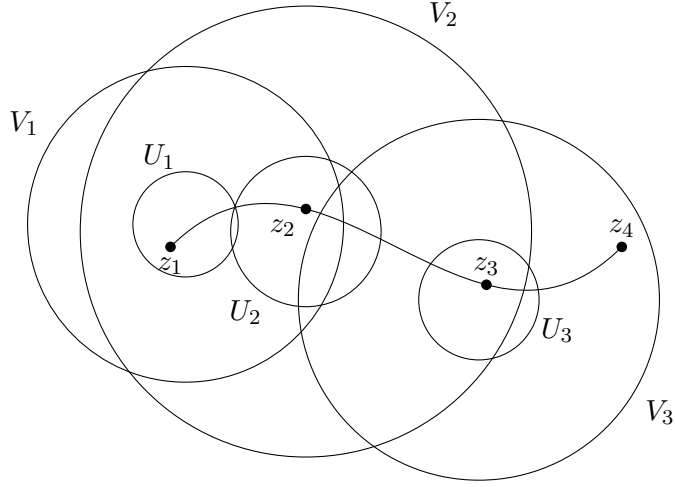

\begin{proof}[Proof of \autoref{thm:main}]
The crux of the proof is that the shortest path length has near sub-\ and super-additivity with high probability.  We will show that if \autoref{thm:Ln-normal-chart} holds locally at every point of $M$, then the local convergences may be assembled together to yield global convergence of the curve length.

For convenience, define $L_n(x,y) = L(x,y; \cX_n)$ in this proof.

For each $w \in M$, we may associate positive $R(w) > 0$ such that \autoref{thm:Ln-normal-chart} holds within open ball $V(w) = B(w; 3R(w))$ with error $2^{-1}\vareps$, i.e., the event $E_n(V(w), 2^{-1}\vareps, b^\prime)$, defined in \autoref{thm:Ln-normal-chart}, satisfies \eqref{eq:normal-chart-expdecay} with any $b^\prime \in (0, 6 R(w))$.
Let $U(w) = B(w; R(w))$.
By compactness, there exists a finite $m > 0$, $\braces{w_i \in M}_{i=1}^{m}$ such that the collection $\braces{U(w_i)}_{i=1}^{m}$ is a finite open cover of $M$.
Define $R_i = R(w_i)$, $U_i = U(w_i)$, and $V_i = V(w_i)$ for $i = 1,\dots,m$.


Reorder the indices if necessary so that $x \in U_1$.
Define $z_1 = x$.
If $L_n(x,y)$ ever exits $V_1$, then a point $z_2 \in V_1$ on the shortest path may be chosen such that $z_2 \notin U_1$ and $\dist_1(z_1,z_2) \geq R_1$.
Note that shortest paths are piece-wise $g_1$-geodesics, and $z_2 \in M$ need not be in $\cX_n$.
Reorder the indices of the open cover again if necessary so that $z_2$ is in $U_2$.  Repeat the procedure until $L_n(x,y)$ ends at $y$ in an open ball, say, $V_k$.  Set $z_{k+1} = y$.  Then points $x = z_1,z_2,\dots,z_k,z_{k+1} = y$ satisfy the conditions $z_i, z_{i+1} \in V_i$ for $i = 1,2,\dots,k$, and $\dist_1(z_i,z_{i+1}) \geq R_i \geq R$ for $i = 1,2,\dots,k-1$, where $R = \min_i R_i$.  The last edge length $\dist_1(z_k,z_{k+1})$ may be less than $R$.  However, note that $z_{k-1} \in U_{k-1}$ and $y = z_{k+1} \notin V_{k-1}$ by definition, hence $\dist_1(z_{k-1},z_{k+1}) > 2 R_{k-1} \geq 2 R$.  Therefore, $z_k$ may be adjusted so that $\dist_1(z_k,z_{k+1}) \geq R$ as well and $z_k \in V_k$.  See \autoref{fig:path-div} for illustration.

Suppose that
\begin{equation}\label{eq:Ln-manifold-cc-eq1}
    (C(d,p) - \vareps) \dist_p(z_i,z_{i+1}) \leq n^{(p-1)/d} L_n(z_i,z_{i+1}),
\end{equation}
holds for all $i=1,2,\dots,k$.
Then by the triangle inequality and the property $\nu^p + \omega^p \leq (\nu+\omega)^p$ for $\nu,\omega \geq 0$ and $p \geq 1$,
\begin{align*}
    (C(d,p) - \vareps) \dist_p(x,y)
    &\leq (C(d,p) - \vareps) \sum_{i=1}^{k} \dist_p(z_i,z_{i+1})
    \\ &\leq \sum_{i=1}^{k} n^{(p-1)/d} L_n(z_i,z_{i+1})
    \\ &\leq n^{(p-1)/d} L_n(x,y).
\end{align*}
Since $m$ is finite and the event $E_n(V_i, 2^{-1}\vareps, \min\braces{R,b})$ from \autoref{thm:Ln-normal-chart} satisfies \eqref{eq:normal-chart-expdecay}, for $i = 1,\dots,m$, we have
\begin{equation*}
  \Pr \paren[3]{\inf_{x,y} \frac{L_n(x,y)}{n^{(1-p)/d} \dist_p(x,y)} < C(d,p) - \vareps}
  \leq m \exp\paren[1]{-\theta_6 n^\alpha}
\end{equation*}
for all sufficiently large $n$.

For the upper tail, we follow a similar strategy to \textcite{Bernstein:2000a}.
Recall that $\alpha = 1/(d+2p)$.
If $z_1 = x$, $z_{k+1} = y$, and $z_i$ are points on a $g_p$-geodesic curve from $x$ to $y$, then $\dist_p(x,y) = \sum_{i=1}^{k} \dist_p(z_i,z_{i+1})$.  We showed above that the points may be chosen and indices of the open cover may be rearranged such that $z_i,z_{i+1} \in V_i$ and $\dist_1(z_i,z_{i+1}) \geq R$ for all $i=1,2,\dots,k$.
The shortest path from $z_{i-1}$ to $z_i$ and another shortest path from $z_i$ to $z_{i+1}$ may be pasted together to create a path from $z_{i-1}$ to $z_{i+1}$ by removing $z_i$ and connecting two nodes that were incident to $z_i$.
This pasting procedure can be repeated to create a path from $x = z_1$ to $y = z_{k+1}$.
Since \autoref{thm:Ln-normal-chart} applies in $V_1,\dots,V_m$, every edge length of the shortest path from $z_i$ to $z_{i+1}$ is at most $b^\alpha (n \inf f)^{(\alpha-1)/d}$ for $i=1,\dots,k$.
Therefore each pasting procedure may incur an additional cost of at most $(2b^\alpha)^p (n \inf f)^{(\alpha-1)p/d}$ so that
\begin{equation}
    L_n(x,y)
    \leq \sum_{i=1}^{k} L_n(z_i,z_{i+1}) + k (2b^\alpha)^p (n \inf f)^{(\alpha-1)p/d}.
\end{equation}
Therefore if event $E_n(V_i,2^{-1}\vareps,\min\braces{R,b})$ in \autoref{thm:Ln-normal-chart} holds for $V_1,V_2,\dots,V_m$, then
\begin{multline*}
  n^{(p-1)/d} L_n(x,y)
  \\
  \begin{aligned}
    &\leq \dist_p(x,y) \paren[2]{C(d,p) + \frac{\vareps}{2}} + k (2b^\alpha)^p n^{(p-1)/d} (n \inf f)^{(\alpha-1)p/d}
  \end{aligned}
\end{multline*}
since $\dist_p(x,y) \leq \dist_1(x,y) (\inf f)^{(1-p)/d}$.
For sufficiently large $n$, we have $n^{(p-1)/d} L_n(x,y) \leq (C(d,p) + \vareps) \dist_p(x,y)$ since $n^{\alpha p - 1}$
shrinks to zero as $n\to\infty$.
Therefore \autoref{thm:main} is established by applications of \autoref{thm:Ln-normal-chart} to $V_1,V_2,\dots,V_m$.
\end{proof}

We turn to the proof of \autoref{thm:main-complete}.  Note that when $M$ is complete, for every $x,y \in M$ there exists a geodesic path $\gamma$ between $x,y$ in $M$ by Hopf-Rinow theorem \autocite[Theorem~5.21, p.~138]{Oneill:1983a}.

\begin{proof}[Proof of \autoref{thm:main-complete}]
Define $L_n(x,y) = L(x,y; \cX_n)$.
Let $0 < \vareps < C(d,p)$.
Define
\begin{equation*}
    A = \braces{ u \in M \colon (C(d,p) - \vareps) \dist_p(x,u) \leq (C(d,p) + \vareps) \dist_p(x,y) }.
\end{equation*}
$A$ is compact by Hopf-Rinow theorem.

The proof of \autoref{thm:main-complete} is similar to the proof of \autoref{thm:main}, with $M$ replaced by $A$.
Let $V_1,\dots,V_m$ be an open cover of $A$ chosen as in the previous proof for compact $M$.

Suppose that the event $E_n(V_i,2^{-1}\vareps,\min\braces{R_1,\dots,R_m, b})$ from \autoref{thm:Ln-normal-chart} holds for $i = 1,\dots,m$.
By the construction of $A$, $g_p$-geodesics from $x$ to $y$ is contained in $A$.
Repeat the same argument used in the proof of \autoref{thm:main} to obtain
\begin{equation}\label{eq:thm-complete-upper}
  n^{(p-1)/d} L_n(x,y) \leq \paren*{C(d,p) + \vareps} \dist_p(x,y).
\end{equation}
Similarly, if the shortest path $L_n(x,y)$ does not exit $A$, then
\begin{equation}\label{eq:thm-complete-lower}
  n^{(p-1)/d} L_n(x,y) \geq (C(d,p) - \vareps) \dist_p(x,y).
\end{equation}

We claim that the shortest path $L_n(x,y)$ does not exit $A$, so that \eqref{eq:thm-complete-lower} is true.
Assume to the contrary that the path $L_n(x,y)$ exits $A$.
Let $u = \argmin_{u^\prime} L_n(x,u^\prime)$ where $u^\prime$ is over all boundary points of $A$.
Since the path $L_n(x,y)$ exits $A$, we have $L_n(x,y) > L_n(x,u)$.
Since the path $L_n(x,u)$ is contained in $A$, \eqref{eq:thm-complete-lower} holds with $u$ in place of $y$.
Since $u$ is a point on the boundary of $A$, we have $(C(d,p) - \vareps) \dist_p(x,u) = (C(d,p) + \vareps) \dist_p(x,y)$.
Combine these with \eqref{eq:thm-complete-upper} to obtain
\begin{align*}
    \paren*{C(d,p) + \vareps} \dist_p(x,y)
    &\geq n^{(p-1)/d} L_n(x,y) > n^{(p-1)/d} L_n(x,u)
    \\ &\geq (C(d,p) - \vareps) \dist_p(x,u)
    \\ &= \paren*{C(d,p) + \vareps} \dist_p(x,y),
\end{align*}
and we have a contradiction.
We have shown that \eqref{eq:thm-complete-lower} holds.

Combine \autoref{thm:Ln-normal-chart} with \eqref{eq:thm-complete-upper} and \eqref{eq:thm-complete-lower} to obtain that
\begin{equation*}
  \Pr \paren*{ \abs*{ \frac{L_n(x,y)}{n^{(1-p)/d} \dist_p(x,y)} - C(d,p) } > \vareps }
\end{equation*}
has exponential decay in $n^\alpha = n^{1/(d+2p)}$.
Almost-sure convergence, and the limit stated in \autoref{thm:main-complete}, follow by the Borel-Cantelli lemma.
\end{proof}



\section{Mean convergence and node cardinality}\label{sec:size-mean-proof}
In this section, we prove \autoref{thm:Ln-size} and \autoref{thm:Ln-mean}.  Since they were stated for sequences $\cX_n$ in a Euclidean space, we return to the Euclidean case $M = \reals^d$.  We introduce a few additional notations used in this section.

The proofs in this section use Poissonization arguments.  We denote by $\cH_\lambda$ a homogeneous Poisson point process in $\reals^d$ of constant intensity $\lambda > 0$.  Specifically, for any Borel set $B$ of Lebesgue measure $\nu(B)$ the cardinality $N_B$ of $\cH_\lambda \cap B$ is a Poisson random variable with mean $\lambda \nu(B)$ and, conditioned on $N_B$, the points $\cH_\lambda \cap B$ are i.i.d.\ uniform over $B$.
We use a shorthand notation $\cL_\lambda(x,y) = L(x,y; \cH_\lambda)$.

Let $e_1 = (1,0,\dots,0) \in \reals^d$ denote the unit vector.
By the translation and rotation invariance of $\cH_\lambda$, the distribution of $\cL_\lambda(x,y)$ for $x,y \in \reals^d$ is the same as the distribution of $\cL_\lambda(0,te_1)$ where $t = \abs{x-y}$.
This observation is used frequently in this section.

Let $T(u,v; b)$ for $u,v \in \reals^d$, $b > 0$, denote the set
\begin{equation}\label{eq:Tuvb-def}
    T(u,v; b) = \bigcup_{0 \leq s \leq 1} B(su + (1-s)v; b).
\end{equation}
Note that $\bigcup_{b > 0} T(u,v; b) = \reals^d$.
For convenience, define
\begin{equation}\label{eq:cL-def}
    \cL_\lambda(u,v; b) = L(u,v; \cH_\lambda \cap T(u,v; b)).
\end{equation}

\subsection{Percolation lemma}
The following lemma on percolation will be used in the proof of \autoref{thm:Ln-size}. 

\begin{lemma}\label{thm:poisson-exit-pr}
Let $\pi$ be a graph path in $\cH_\lambda$ starting at $0 \in \reals^d$.  Suppose that $\pi$ has power-weighted path length at most $c_0 \lambda^{(1-p)/d}$ and has at least $c_1 \lambda^{1/d}$ nodes for some $c_0,c_1 > 0$.  Then there exists a constant $\rho_0 > 0$, dependent on $d$ and $p$, such that if $c_1 > \rho_0 c_0$ then the probability that such path $\pi$ exists is exponentially small in $c_1 \lambda^{1/d}$.
\end{lemma}

\begin{proof}
The structure of the proof is similar to that of \textcite[Theorem~6.1]{Meester:1996a}.  We first define a Galton-Watson process $\BbbX_n$.  Let $\BbbX_0 = \{ x_0 = 0 \in \reals^{d} \}$ be the ancestor of the family, and associate the parameter $r_0 > 0$.  Then define the offspring $\BbbX_1(r_0)$ to be $\cH_{\lambda} \cap B(x_0; r_0^{1/p})$.  $\BbbX_1(r_0)$ is the set of points in $\cH_{\lambda}$ that may be reached from $x_0$ with a single edge with path length at most $r_0$ in power-weighted sense.  Note that $\BbbE \abs{\BbbX_1(r_0)} = \lambda V_d r_0^{d/p}$ where $\abs{\BbbX_1(r_0)}$ denotes the cardinality of $\BbbX_1(r_0)$, and $V_d$ denotes the volume of $B(0; 1)$.

For each offspring $x_{1,k} \in \BbbX_1(r_0)$, we associate the parameter $r_{1,k} = r_0 - \abs{x_{1,k} - x_0}^p$.  Then $\cH_{\lambda}$ in the union of $B(x_{1,k}; r_{1,k}^{1/p}) - \braces{x_{1,k}}$ over $k$ is the set of points that may be reached from $x_0$ with exactly two edges, while the power-weighted path length is at most $r_0$.  Define $\BbbX_2(r_0)$ to be the collection of all the second generation offspring, and define recursively the $n$-th generation offspring $\BbbX_n(r_0)$.  Then $\BbbX_n(r_0)$ is the set of all the points that may be reached in $n$ hops from the ancestor $x_0$ within path length $r_0$.  See \autoref{fig:percolation-treerun}.  We prove by induction that
\begin{equation}\label{eq:poisson-percolation-claim}
    \BbbE \abs[1]{\BbbX_n(r_0)} 
    \leq \paren[1]{ \lambda V_d r_0^{d/p} }^n \frac{\Gamma(1 + d/p)^n}{\Gamma(1 + nd/p)}.
\end{equation}

We mentioned above that $\BbbE \abs{\BbbX_1(r_0)} = \lambda V_d r_0^{d/p}$, and \eqref{eq:poisson-percolation-claim} is true for $n = 1$.  For general $n$, apply the Campbell-Mecke formula \autocite[Theorem 3.2, p.48]{Baddeley:2007a} to see that
\begin{align}
    \BbbE \abs[1]{\BbbX_n(r_0)} 
    &\leq \lambda \int_{B(x_0; r_0^{1/p})} \BbbE \abs[1]{ \BbbX_{n-1}\paren[1]{r_0 - \abs{x - x_0}^p} } \,dx \notag
    \\ &\leq \lambda^n V_d^{n-1} \frac{\Gamma(1 + d/p)^{n-1}}{\Gamma(1 + (n-1)d/p)} \int_{B(x_0; r_0^{1/p})} \paren[1]{r_0 - \abs{x-x_0}^p}^{(n-1)d/p} \,dx.
    \label{eq:poisson-percolation-campbell}
\end{align}
The last integral evaluates to
\begin{align*}
    &\phantom{=} \int_{B(x_0; r_0^{1/p})} \paren[1]{r_0 - \abs{x - x_0}^p}^{(n-1)d/p} \,dx
    \\ &= V_d r_0^{(n-1)d/p} d \int_0^{r_0^{1/p}} \paren[2]{1 - \frac{u^p}{r_0}}^{(n-1)d/p} u^{d-1} \,du
    \\ &= V_d r_0^{nd/p} \frac dp \int_0^1 (1 - v)^{(n-1)d/p} v^{d/p - 1} \,dv
    \\ &= V_d r_0^{nd/p} \frac{\Gamma(1 + d/p) \Gamma(1 + (n-1)d/p)}{\Gamma(1 + nd/p)}
\end{align*}
Note that a spherical coordinate transformation was used in the first equality, a transformation $v = u^p / r_0$ was used in the second equality, and the third equality was obtained by properties of the beta function.  Substituting the expression in the last line into \eqref{eq:poisson-percolation-campbell} establishes \eqref{eq:poisson-percolation-claim}.

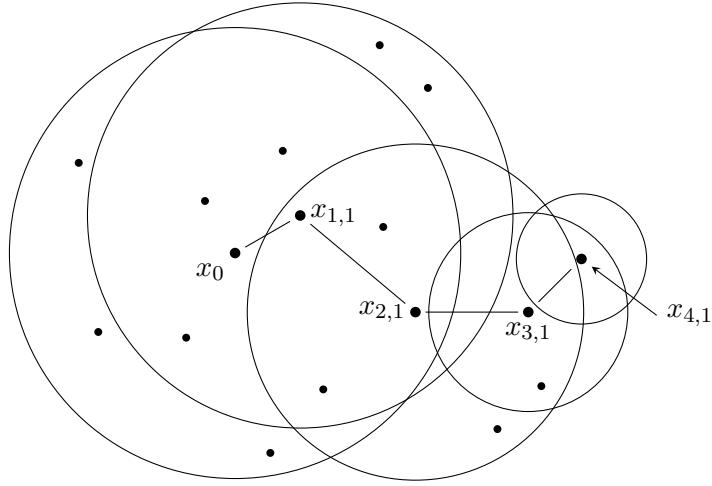
\begin{figure}
\centering
\begin{tikzpicture}
    \path (0,0) node (xzero) {};
    \path (xzero)  -- +(30:1)  node (xone)   {};
    \path (xone)   -- +(-40:2) node (xtwo)   {};
    \path (xtwo)   -- +(0:1.5) node (xthree) {};
    \path (xthree) -- +(45:1)  node (xfour)  {};

    \draw (xzero)  circle (3);
    \draw (xone)   circle (2.8284);
    \draw (xtwo)   circle (2.2361);
    \draw (xthree) circle (1.3229);
    \draw (xfour)  circle (0.8660);

    \fill (xzero)  circle (2pt) node [below left] {$x_0$};
    \fill (xone)   circle (2pt) node [right]      {$x_{1,1}$};
    \fill (xtwo)   circle (2pt) node [left]       {$x_{2,1}$};
    \fill (xthree) circle (2pt) node [below]      {$x_{3,1}$};
    \fill (xfour)  circle (2pt);
    \draw [stealth-] (xfour) -- +(1,-.75) node [right] {$x_{4,1}$};

    \draw [thin] (xzero) -- (xone) -- (xtwo) -- (xthree) -- (xfour);

    \fill (10:2)    circle (1.5pt);
    \fill (65:1.5)  circle (1.5pt);
    \fill (120:0.8) circle (1.5pt);
    \fill (150:2.4) circle (1.5pt);
    \fill (240:1.3) circle (1.5pt);
    \fill (210:2.1) circle (1.5pt);
    \fill (280:2.7) circle (1.5pt);
    \fill (xone) +(45:2.4)  circle (1.5pt);
    \fill (xone) +(65:2.5)  circle (1.5pt);
    \fill (xtwo) +(-55:1.9) circle (1.5pt);
    \fill (xtwo) +(220:1.6) circle (1.5pt);
    \fill (xthree) +(-80:1.0) circle (1.5pt);
\end{tikzpicture}
\caption{A run through the family tree generated by $\BbbX_n$ with $p = 2$.  The point $x_0$ is the ancestor with parameter $r_0 = 9$.  This means that all the runs through the family tree are paths with power-weighted length less than $r_0^{1/p} = 3$.  Here $x_{1,1} \in \BbbX_1$ is among the first generations since it is within $B(x_0; r_0^{1/p})$, and $x_{2,1} \in \BbbX_2$ is among the second generations since it is within the balls centered at the first generation offsprings, e.g., $x_{1,1}$.  This particular run ends at $x_{4,1}$ as there is no point in the vicinity.  In this example, the power-weighted path length is $\sqrt{1^2+2^2+1.5^2+1^2} = \sqrt{8.25} < 3$.  Note that $x_{2,1}$ is also in the ball centered at $x_0$, so it is also a first generation offspring.  Some other runs through the family tree will have the point $x_{2,1}$ as a first generation offspring.}
\label{fig:percolation-treerun}
\end{figure}

Using the Markov inequality and Stirling's approximation, we have
\begin{equation*}
    \log \Pr \paren{\BbbX_n(r_0) \neq \emptyset}
    \leq n \log \paren[3]{ V_d \Gamma \paren[2]{1 + \frac{d}{p}} \paren[2]{\frac{c_0}{c_1} \cdot \frac{p e}{d}}^{d/p} } + O(\log n)
\end{equation*}
as $n\to\infty$.
Note that if a path starting at $x$ passes through more than $n \geq c_1 \lambda^{1/d}$ nodes and has path length less than $r_0 \leq c_0 \lambda^{(1-p)/d}$, then the $n$-th generation set $\BbbX_n(r_0)$ will not be empty.
\autoref{thm:poisson-exit-pr} follows since, if the ratio $c_1 / c_0$ is sufficiently large, the logarithm term above is negative.
\end{proof}


\subsection{Mean convergence for Poisson point processes}
\begin{lemma}\label{thm:poisson-mean}
Consider the shortest path length $\cL_1(0,te_1)$ from $0 \in \reals^d$ to $te_1 \in \reals^d$ in $\cH_1$ for $t > 0$.
Then the limit
\begin{equation}\label{eq:poisson-shape}
    \lim_{t\to\infty} \frac1t \BbbE \cL_1(0,te_1) = C(d,p)
\end{equation}
exists.

In addition, if $b = b_t$ is a function of $t$ satisfying $\liminf_t b_t = \infty$, then
\begin{equation}\label{eq:poisson-shape-local}
    \lim_{t\to\infty} \frac1t \BbbE \cL_1(0,te_1; b_t) = C(d,p).
\end{equation}
\end{lemma}

Recall that $\cL_1(0,te_1; b_t)$ denotes $L(0,te_1; \cH_1 \cap T(0, te_1; b_t))$ from \eqref{eq:Tuvb-def} and \eqref{eq:cL-def}.

When $T(0, te_1; b) = \reals^d$, i.e., $b = +\infty$, \eqref{eq:poisson-shape} is a consequence of, e.g., \textcite[Section~4]{Howard:2001a}.
The main difference is the case when $b < +\infty$.
\textcite[Theorem~2.4]{Howard:2001a} states that the probability that $\cL_1(0,te_1) \neq \cL_1(0,te_1; b_t)$ is exponentially small of order at least $t^{3p\vareps/4}$ when $b_t \geq t^{3/4 + \vareps}$ for some $\vareps > 0$.
\autoref{thm:poisson-mean} is weaker in the sense that it only asserts closeness in the mean.
On the other hand, \autoref{thm:poisson-mean} is stronger in the sense that the assumption on $b_t$ is relaxed so that $b_t$ need only diverge to infinity, and the rate of growth may even be sub-polynomial.

\begin{proof}[Proof of \autoref{thm:poisson-mean}]
Initially we let $b > 0$ be a constant instead of a function of $t$.
This assumption is removed later in the proof.
Recall the definition of function $h$ in \eqref{eq:h-def},
\begin{equation*}
    h(x,y; u) = \abs{x-u}^p + \abs{u-y}^p - \abs{x-y}^p.
\end{equation*}
Let
\begin{equation*}
    T(b) = \bigcup_{s > 0} T(-se_1, +se_1; b),
\end{equation*}
and let
\begin{equation*}
  \xi_t(\lambda, b) = \sup \braces[2]{ \abs{u-te_1} \colon u \in T(b),\, h(u,te_1; v) \geq 0 \text{ for all $v \in \cH_\lambda \cap T(b)$} }.
\end{equation*}
In other words, $\xi_t(\lambda, b)$ denotes an upper bound distance of $u \in T(b)$ from $te_1$ such that the shortest path from $te_1$ to $u$ is the direct path $te_1 \to u$.
From the continuity of function $h$, it is not difficult to show that there exist constants $A,\delta > 0$ and constant integers $k,m > 0$, all independent of $b$ and $\lambda$, such that for all $t \in \reals$,
\begin{equation}\label{eq:xi-upperbound}
    \BbbE \xi_t(\lambda, b)^p
    \leq \frac{k \Gamma(1 + p/d)}{(\lambda A)^{p/d}} + \frac{m 2^p \Gamma(1+p)}{\lambda^p (\delta b)^{p(d-1)}}.
\end{equation}
It is not surprising that the upper bound does not depend on $t$ since $\cH_\lambda$ is homogeneous.
For a simple proof of this see \textcite[Lemma~2.5, Equation~2.14]{Hwang:2012b}.

\begin{figure}
\centering
\begin{tikzpicture}[scale=4]
    \draw (0,0) node [anchor=north east] {$0$} -- (.3,-.3) -- (.5,.2) -- (.8,.4) node [above] {$\gamma_{-}$};
    \draw [dashed] (.8,.4) -- (1,0) node [below] {$\phantom{(}se_1 = \gamma_0\phantom{)}$} -- (1.2,.2) node [anchor=south west] {$\gamma_{+}$};
    \draw (1.2,.2) -- (1.4,-.2) -- (1.6,-.4) -- (1.9,.2) -- (2,0) node [anchor=north west] {$(s+t)e_1$};
    \draw [very thick] (.8,.4) -- (1.2,.2);

    \draw [-latex] (-.5,0) -- (2.5,0);
    \fill [black,opacity=.5] (0,0) circle (.5pt) (.3,-.3) circle (.5pt) (.5,.2) circle (.5pt) (.8,.4) circle (.5pt) (1,0) circle (.5pt);
    \fill [black,opacity=.5] (1.2,.2) circle (.5pt) (1.4,-.2) circle (.5pt) (1.6,-.4) circle (.5pt) (1.9,.2) circle (.5pt) (2,0) circle (.5pt);
\end{tikzpicture}
\caption{An illustration of the path pasting procedure.  A new path from $0$ to $(s+t)e_1$ is created by removing $se_1 = \gamma_0$ and joining $\gamma_{-}$ and $\gamma_{+}$.  Only the end points are fixed points in the new path.}
\label{fig:path-paste}
\end{figure}
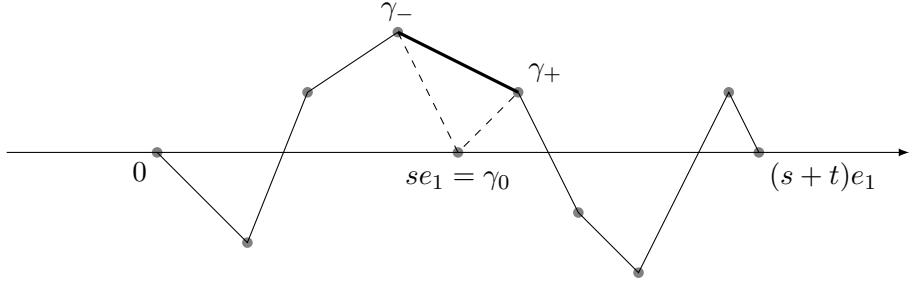

Let $s,t > 0$.  Consider the shortest path $\cL_1(0,se_1; b)$ between $0$ and $se_1$, and let $\gamma_{-}$ denote the node that directly connects to $se_1$.  Similarly consider the shortest path for $\cL_1(se_1,(s+t)e_1; b)$ and let $\gamma_{+}$ denote the node that directly connects to $se_1$.  Therefore $\gamma_{-}$ and $\gamma_{+}$ are Poisson sample points incident to $se_1$.  For convenience let $\gamma_0 = se_1$.  Remove $\gamma_0 = se_1$ in the two paths, and join the nodes $\gamma_{-}$ and $\gamma_{+}$ so that we have a new path connecting $0$ and $(s+t)e_1$, as indicated in \autoref{fig:path-paste}.  This new path has length that is an upper bound on $\cL_1(0,(s+t)e_1; b)$,
\begin{equation*}
    \cL_1(0,(s+t)e_1; b)
    \leq \cL_1(0,se_1; b) + \cL_1(se_1; (s+t)e_1; b)
        + (\abs{\gamma_0 - \gamma_{-}} + \abs{\gamma_{+} - \gamma_0})^p.
\end{equation*}
Note that both $\abs{\gamma_0 - \gamma_{-}}$ and $\abs{\gamma_{+} - \gamma_0}$ are bounded above by $\xi_s(1, b)$, and $\BbbE \xi_s(1,b)^p$ is finite by \eqref{eq:xi-upperbound}.
Therefore $\BbbE \cL_1(0,(s+t)e_1; b)$ is bounded above by
\begin{align*}
    &\BbbE \cL_1(0,se_1; b) + \BbbE \cL_1(se_1; (s+t)e_1; b) + \BbbE (2 \xi_s(1, b))^p
    \\ = {} & \BbbE \cL_1(0,se_1; b) + \BbbE \cL_1(0; te_1; b) + \BbbE (2 \xi_0(1, b))^p.
\end{align*}
The equality holds by the translation invariant property of the distribution of $\cH_1$.  Therefore $\BbbE \cL_1(0,te_1; b) + \BbbE (2 \xi_0(1, b))^p$ is a sub-additive function of $t$.  Note that $\BbbE \cL_1(0,te_1; b) \leq t^p$.  A standard proof of Fekete's lemma \autocite[for example, see][Lemma~1.2.1]{Steele:1997a} may be easily adapted to sub-additive functions that are bounded in bounded intervals.  Apply Fekete's lemma to the sub-additive function $\BbbE \cL_1(0,te_1; b) + \BbbE (2 \xi_0(1, b))^p$, then
\begin{equation}\label{eq:poisson-mean-fekete}
    \lim_{t\to\infty} \frac{\BbbE \cL_1(0,te_1; b) + \BbbE (2 \xi_0(1, b))^p}{t}
    = \inf_{t>0} \frac{\BbbE \cL_1(0,te_1; b) + \BbbE (2 \xi_0(1, b))^p}{t},
\end{equation}
and we denote the limit by $\kappa(d,p; b)$.  Note that $\BbbE \xi_0(1,b)^p$ does not depend on $t$, hence \eqref{eq:poisson-mean-fekete} implies that $\lim_t t^{-1} \BbbE \cL_1(0,te_1; b) = \kappa(d,p; b)$.  

Define
\begin{equation}\label{eq:cdp-definition}
    C(d,p) = \lim_{t\to\infty} \frac{\BbbE \cL_1(0,te_1)}{t}.
\end{equation}

We now show that $\kappa(d,p; b)$ converges to $C(d,p)$ when $b\to\infty$.  
Choose an arbitrary $\vareps > 0$.
By \eqref{eq:xi-upperbound} and by the fact that $C(d,p)$ is the limit of $t^{-1} \BbbE \cL_1(0,te_1)$, there exists $T > 0$ such that
\begin{equation*}
    \frac1T \BbbE \cL_1(0,Te_1) < C(d,p) + \frac{\vareps}{3},
\end{equation*}
and
\begin{equation*}
    \frac1T \BbbE (2 \xi_0(1, b))^p < \frac{\vareps}{3},
\end{equation*}
for all $b > 1$.  For this fixed $T$, note that $\lim_{b\to\infty} \cL_1(0,Te_1; b) = \cL_1(0,Te_1)$ monotonically from above almost surely, and by the monotone convergence theorem, there exists $B > 1$ such that for all $b > B$ and fixed $T$,
\begin{equation*}
    \frac1T \BbbE \cL_1(0,Te_1; b)
    \leq \frac1T \BbbE \cL_1(0,Te_1) + \frac{\vareps}{3}.
\end{equation*}
Combining the three inequalities above with \eqref{eq:poisson-mean-fekete} we obtain
\begin{equation*}
    \kappa(d,p; b)
    \leq \frac1T \paren[1]{\BbbE \cL_1(0,Te_1; b) + \BbbE (2 \xi_0(1, b))^p}
    \leq C(d,p) + \vareps,
\end{equation*}
for all $b > B$.  Therefore $\kappa(d,p; b)$ converges to $C(d,p)$ as $b\to\infty$.

Finally, suppose $b = b_t$ is a function of $t$ rather than a constant.  If $\liminf_t b_t = \infty$ then
\begin{align*}
    C(d,p)
    &\leq \lim_{t\to\infty} \frac1t \BbbE \cL_1(0, te_1; b_t)
    \\ &\leq \lim_{t\to\infty} \frac1t \BbbE \cL_1(0, te_1; B)
    = \kappa(d,p; B),
\end{align*}
for any fixed $B > 0$.  \eqref{eq:poisson-shape-local} follows as $B\to\infty$ on the right side.
\end{proof}

For the readers' benefit we establish two use cases of \autoref{thm:poisson-mean}.

\begin{corollary}
\label{thm:poisson-mean-cor}
The following two cases follow from \autoref{thm:poisson-mean}.

  \begin{enumerate}[label=\textup{(\roman*)}]
    \item \label{it:poisson-mean-i1}
    For every $\vareps > 0$ there exists a constant $t_0 > 0$ such that for all $\lambda > 0$ and $r > 0$ satisfying $\lambda^{1/d} r > t_0$,
    \begin{equation*}
      \abs*{\frac{\BbbE \cL_\lambda(0, r e_1; r)}{r \lambda^{(1-p)/d}} - C(d,p)} < \vareps.
    \end{equation*}

    \item \label{it:poisson-mean-i2} Let $z \in \reals^d$ and $R_2 > R_1 > 0$.
    Let $b > 0$ and $\vareps > 0$.
    Then there exists $\lambda_0 > 0$ such that for all $\lambda \geq \lambda_0$ and $x,y \in B(z; R_1)$ with $\abs{x-y} \geq b$, we have
    \begin{equation*}
      \abs*{\frac{\BbbE L(x,y; \cH_\lambda \cap B(z; R_2))}{\lambda^{(1-p)/d} \abs{x-y}} - C(d,p)} < \vareps.
    \end{equation*}
\end{enumerate}
\end{corollary}

\begin{proof}
$\cH_1$ scaled by factor of $\lambda^{-1/d}$ has identical distribution to $\cH_\lambda$.
At the same time, power-weighted shortest path lengths are scaled by factor of $\lambda^{-p/d}$.
From \eqref{eq:poisson-shape-local}, we have
\begin{equation}\label{eq:poisson-mean-scale}
  \frac{\BbbE \cL_\lambda(0, \lambda^{-1/d} te_1; \lambda^{-1/d} b_t)}{\lambda^{(1-p)/d} \lambda^{-1/d} t}
  = \frac1t \BbbE \cL_1(0, te_1; b_t).
\end{equation}
Choose $b_t = t$ and $t = \lambda^{1/d} r$ to obtain \autoref{it:poisson-mean-i1} from \autoref{thm:poisson-mean}.

For \autoref{it:poisson-mean-i2}, note that
\begin{equation*}
  \cL_\lambda(x,y)
  \leq L(x,y; \cH_\lambda \cap B(z; R_2))
  \leq \cL_\lambda(x,y; R_2 - R_1)
\end{equation*}
since $x,y \in B(z; R_1)$.
By translation- and rotation-invariance of $\cH_\lambda$,
\begin{equation*}
  \BbbE \cL_\lambda(0,\abs{x-y} e_1)
  \leq \BbbE L(x,y; \cH_\lambda \cap B(z; R_2))
  \leq \BbbE \cL_\lambda(0,\abs{x-y} e_1; R_2 - R_1).
\end{equation*}
Choose $t = \lambda^{1/d} \abs{x-y}$ and $b_t = \lambda^{1/d} (R_2 - R_1)$ for \eqref{eq:poisson-mean-scale}.
Then \autoref{it:poisson-mean-i2} follows from \autoref{thm:poisson-mean}.
\end{proof}

\subsection{Shortest path size}
In order to prove \autoref{thm:Ln-size}, we need an upper bound for shortest path lengths in $\cH_\lambda$.

\begin{lemma}\label{thm:poisson-upper-bound}
Let $z \in \reals^d$, $R_2 > R_1 > 0$.
Let $b > 0$ and $\vareps > 0$.
For every $x,y \in B(z; R_1)$, let $E_\lambda(x,y,\vareps)$ denote the event that
\begin{equation}\label{eq:poisson-upper-bound-edef}
    \frac{L(x,y; \cH_\lambda \cap B(z; R_2))}{\lambda^{(1-p)/d} \abs{x-y}} \leq C(d,p) + \vareps.
\end{equation}
Then there exist $\lambda_0 > 0$ and $\theta_7 > 0$ such that for all $\lambda \geq \lambda_0$ and $\abs{x-y} \geq b$,
\begin{equation}
\label{eq:poisson-upper-bound}
  1 - \Pr(E_\lambda(x,y,\vareps))
  \leq \exp\paren*{ -\theta_7 \lambda^{1/(d+2p-1)} }.
\end{equation}
\end{lemma}

\begin{proof}
Let $0 < r < R_2 - R_1$.
Recall the notation $T(x,y; r)$ and $\cL_\lambda(x,y; r)$ from \eqref{eq:Tuvb-def} and \eqref{eq:cL-def}.
Since $x,y \in B(z; R_1)$ and $r < R_2 - R_1$,
\begin{equation*}
    \cH_{\lambda} \cap T(x,y; r)
    \subset \cH_\lambda \cap B(z; R_2),
\end{equation*}
and hence
\begin{equation*}
    L(x,y; \cH_\lambda \cap B(z; R_2))
    \leq \cL_\lambda(x,y; r).
\end{equation*}
Let $E^\prime_\lambda(x,y,\vareps)$ denote the event that \eqref{eq:poisson-upper-bound-edef} holds with $\cL_\lambda(x,y; r)$ in place of $L(x,y; \cH_\lambda \cap B(z; R_2))$.
By the inequality above, if $E^\prime_\lambda(x,y,\vareps)$ occurs then $E_\lambda(x,y,\vareps)$ occurs, hence $\Pr(E^\prime_\lambda(x,y,\vareps)) \leq \Pr(E_\lambda(x,y,\vareps))$.
Therefore it is sufficient to show that \eqref{eq:poisson-upper-bound} holds with $1 - \Pr(E^\prime_\lambda(x,y,\vareps))$ in place of $1 - \Pr(E_\lambda(x,y,\vareps))$.

As in \autoref{thm:poisson-mean}, by the convex property of the power functions, $\cL_\lambda(0,2re_1; r)$ may be bounded above by $\cL_\lambda(0,re_1; r) + \cL_\lambda(re_1,2re_1; r) + (2^{p-1} - 1) (Z_1^p + Y_0^p)$, where $Z_k$ and $Y_k$ are the first and the last edge lengths in $\cL_\lambda(kre_1,(k+1)re_1; r)$, respectively.  In \autoref{fig:path-paste}, when $s = r$ and $s+t = 2r$, $Z_1$ and $Y_0$ correspond to $\abs{\gamma_{+} - \gamma_0}$ and $\abs{\gamma_0 - \gamma_{-}}$, respectively.

Note that the shortest path for $\cL_\lambda(kre_1,(k+1)re_1; r)$ is not likely to be the direct path $kre_1 \to (k+1)re_1$.  That is, if it were the direct path, then as in the proof of \autoref{thm:euc-shortlink}, there exists $\delta > 0$ such that $\cH_\lambda$ is empty in the open ball of radius $\delta r$ centered at the middle of $kre_1$ and $(k+1)re_1$.
Such event happens with probability at most $\exp( -\lambda \theta^\prime r^d )$, where $\theta^\prime$ denotes the volume of an open ball of radius $\delta$.
If none of the shortest paths for $\cL_\lambda(kre_1, (k+1)re_1; r)$ is a direct path, then the previous pasting procedure used in \autoref{thm:poisson-mean} may be repeated so that
\begin{equation}\label{eq:poisson-limsup-eq1}
    \cL_\lambda(0,mre_1; r)
    \leq \sum_{k=0}^{m-1} \paren[2]{ \cL_\lambda\paren[1]{kre_1,(k+1)re_1; r} + (2^{p-1} - 1) (Z_k^p + Y_k^p) },
\end{equation}
with probability at least $1 - m \exp( - \lambda \theta^\prime r^d)$.

If $k,l$ are integers and $l - k \geq 3$, then $T(kre_1, (k+1)re_1; r)$ and $T(lre_1, (l+1)re_1; r)$ are disjoint, hence $\cL_\lambda(kre_1,(k+1)re_1; r)$ and $\cL_\lambda(lre_1, (l+1)re_1; r)$ are mutually independent, and so are $Z_k$ and $Z_l$, as well as $Y_k$ and $Y_l$.  Then the sum in \eqref{eq:poisson-limsup-eq1} may split into $K \geq 3$ sums of independent variables, and each sum has at least $\floor{m / K}$ summands.  Note that each summand is almost surely bounded since $Z_k^p + Y_k^p \leq \cL_\lambda(kre_1, (k+1)re_1; r) \leq r^p$.  Apply Azuma-Hoeffding's inequality \autocite{Azuma:1967a} for $K = 4$ separate sequences to obtain
\begin{multline}\label{eq:poisson-upper-bound-oneeps}
    \Pr \paren*{ \frac{ \cL_\lambda(0,mre_1; r) }{ \lambda^{(1-p)/d} mr } \geq \mu_r + \vareps }
    \\ \leq m \exp\paren[1]{-\lambda \theta^\prime r^d}
        + 4 \exp\paren*{ -\frac{(m-3) \vareps^2}{2^{1+2p} (\lambda^{1/d} r)^{2(p-1)}} },
\end{multline}
where
\begin{equation*}
  \mu_r = \frac{\BbbE \cL_\lambda(0,re_1; r) + (2^{p-1} - 1) (\BbbE Z_0^p + \BbbE Y_0^p)}{r \lambda^{(1-p)/d}}.
\end{equation*}
Let $\beta = 1/(d+2p-1)$.
Set $mr = \abs{x-y}$ and $m = \floor{ (\lambda^{1/d} \abs{x-y})^{1-\beta} }$.
Note that $r = \abs{x-y} / m \leq 2 R_1 / m$ is less than $R_2 - R_1$ when $\lambda^{1/d} \abs{x-y}$ is sufficiently large.
By the definition, both $\BbbE Z_k^p$ and $\BbbE Y_{k-1}^p$ are bounded above by $\BbbE \xi_{kb}(\lambda, r)^p = \BbbE \xi_0(\lambda, r)^p$ in \eqref{eq:xi-upperbound}, and a direct computation with \eqref{eq:xi-upperbound} shows that $\BbbE \xi_0(\lambda, r)^p$ divided by $\lambda^{(1-p)/d} r$ shrinks to zero when $\lambda^{1/d} r \geq (\lambda^{1/d} \abs{x-y})^\beta \to \infty$.
See \textcite[Lemma~2.5]{Hwang:2012b} for more details.  Apply \autoref{thm:poisson-mean-cor} to see that $\mu_r$ converges to $C(d,p)$ as $\lambda^{1/d} r \to\infty$.  Then \eqref{eq:poisson-upper-bound-oneeps} becomes
\begin{multline*}
  \Pr \paren*{ \frac{ \cL_\lambda(x,y; r) }{ \lambda^{(1-p)/d} \abs{x-y} } \geq C(d,p) + 2 \vareps }
  \\
  \begin{aligned}
    &\leq \paren[1]{\lambda \abs{x-y}^d}^{(1-\beta)/d} \exp\paren*{-\theta^\prime (\lambda \abs{x-y}^d)^{\beta}}
      + 4 \exp\paren*{ -\frac{ (\lambda \abs{x-y}^d)^{\beta}}{ 2^{2(p+1)} } \vareps^2},
    \\ &\leq \paren[1]{\lambda (2R_1)^d}^{(1-\beta)/d} \exp\paren*{-\theta^\prime (\lambda b^d)^{\beta}}
      + 4 \exp\paren*{ -\frac{ (\lambda b^d)^{\beta}}{ 2^{2(p+1)} } \vareps^2},
  \end{aligned}
\end{multline*}
for all sufficiently large $\lambda$.
\end{proof}

\begin{proof}[Proof of \autoref{thm:Ln-size}]
Fix constants $A > 1$ and $0 < A' < 1$.
Let $N$ and $N'$ be independent Poisson variables with mean $nA$ and $nA'$, respectively.
Let $a = Af(z)$ and $a' = A'f(z)$.
Let $H_n$ denote the event that $N \geq n$ and $N' \leq n$.
Let $K_n$ denote the event that
\begin{equation}\label{eq:Ln-size-Kn-def}
    \frac{L(x,y; \cH_{na'} \cap B(z; R_2))}{(na')^{(1-p)/d} \abs{x-y}}
    \leq C(d,p) + \frac{\vareps}{2}.
\end{equation}

We first show that if both $H_n$ and $K_n$ occur, then the following conditions are satisfied.
\begin{enumerate}[label=(\roman*)]
  \item \label{it:Ln-size-it1} $L(x,y; \cX_n \cap B(z; R_2))$ is a path in $\cH_{na}$.
  \item \label{it:Ln-size-it2} $L(x,y; \cX_n \cap B(z; R_2)) \leq (C(d,p) + \vareps/2) (n a')^{(1-p)/d} \abs{x-y}$.
\end{enumerate}

Note that restriction of $\cH_{na}$ to $B(z; R_2)$ may be realized as $\cX_N \cap B(z; R_2)$ since $\BbbE N = nA$.
Since $H_n$ is assumed to occur, it follows that $N \geq n$, and $\cX_n \cap B(z; R_2) \subset \cX_N \cap B(z; R_2) = \cH_{na} \cap B(z; R_2) \subset \cH_{na}$.
Therefore \autoref{it:Ln-size-it1} holds.

For condition \autoref{it:Ln-size-it2}, $H_n$ is assumed to occur, so we have $N' \leq n$.
Then similar to the previous argument, $\cH_{na'} \cap B(z; R_2) = \cX_{N'} \cap B(z; R_2) \subset \cX_n \cap B(z; R_2)$ and it follows that $L(x,y; \cX_n \cap B(z; R_2)) \leq L(x,y; \cH_{na'} \cap B(z; R_2))$.
Condition \autoref{it:Ln-size-it2} follows by \eqref{eq:Ln-size-Kn-def}.

Recall that $G_n = G_n(x,y)$ denotes the event $\# L(x,y; \cX_n \cap B(z; R_2))$ is less than or equal to  $C_* (nf(z))^{1/d} \abs{x-y}$.
We have shown that when $H_n$ and $K_n$ occur, \autoref{it:Ln-size-it1} and \autoref{it:Ln-size-it2} hold, and an application of \autoref{thm:poisson-exit-pr} shows that
\begin{equation*}
  1 - \Pr\paren{G_n \,\delimsize\vert\, H_n \cap K_n}
  \leq \exp\paren[1]{-C (nf(z))^{1/d}\abs{x-y}}
  \leq \exp\paren[1]{-C^\prime n^{1/d}}
\end{equation*}
for some $C,C^\prime > 0$ when $C_* > (C(d,p) + \vareps/2) A^{p/d} A'^{(1-p)/d} \rho_0$.
See \autoref{thm:poisson-exit-pr} for the constant $\rho_0$.

By \autoref{thm:poisson-upper-bound}, $1 - \Pr(K_n)$ is bounded above by $\exp(-\theta_7 (na^\prime)^{1/(d+2p-1)})$ for sufficiently large $n$, since $a^\prime > 0$ is a fixed constant.
By the Chernoff bound \autocite[Theorem~9.3]{Billingsley:1995a}, $1 - \Pr(H_n)$, i.e., the probability that either $N < n$ or $N' > n$, is exponentially small in $n$.

Note that $1 - \Pr(G_n)$ is bounded above by the sum of $(1 - \Pr(G_n \mid H_n \cap K_n))$, $(1 - \Pr(H_n))$, and $(1 - \Pr(K_n))$.
The lemma follows from the observation that the overall decay is determined by the summand with slowest decay rate, and it is $(1- \Pr(K_n))$, which is exponentially small in $n^{1/(d+2p-1)}$.
\end{proof}

\subsection{Mean convergence in i.i.d.\ cases: de-Poissonization}
\begin{proof}[Proof of \autoref{thm:Ln-mean}]
For convenience, let $L_k$ denote $L(x,y; \cX_k \cap B(z; R_2))$ for all $k \geq 0$.
Recall that $\alpha = 1/(d+2p)$.
Let $\tau_k = (k f(z))^{1/d} \abs{x-y}$, and $\zeta_k = (k f(z))^{(\alpha-1)/d} \abs{x-y}^\alpha$.

Let $C_* > 0$ as in \autoref{thm:Ln-size} and suppose that the number of nodes $\# L_k$ in the shortest path $L_k$ is less than $C_* \tau_k$.
Suppose that the event $F_k(\abs{x-y})$ from \autoref{thm:euc-shortlink} occurred so that all the shortest path edge lengths are at most $\zeta_k$.
When a sample point from $\cX_k$ is discarded, $L_{k-1}$ remains the same as $L_k$ if the discarded sample point were not a node in $L_k$.
Furthermore since edge lengths are at most $\zeta_k$, $L_{k-1}$ and $L_k$ may differ at most by $(2 \zeta_k)^p$.
Therefore
\begin{equation}\label{eq:Ln-mean-eq1}
    \BbbE L_{k-1} - \BbbE L_k
    \leq \frac{C_* \tau_k}{k} (2 \zeta_k)^p + h_k \BbbE L_0,
\end{equation}
where $h_k$ denotes the probability that either $\# L_k > C_* \tau_k$, or the event $F_k(\abs{x-y})$ does not occur.
$\BbbE L_0$ in the last term is chosen because $\BbbE L_k \leq \BbbE L_0$ for all $k > 0$.

Let $N$ be a Poisson variable with mean $n$.
Write
\begin{equation*}
    \BbbE L_N = \sum_{k \geq 0} \BbbE L_k \Pr \paren{N = k}.
\end{equation*}
The difference $\abs{\BbbE L_n - \BbbE L_N}$ is bounded above by
\begin{equation*}
    \sum_{k \geq 0} \abs{\BbbE L_n - \BbbE L_k} \Pr \paren{N = k}
    \leq \BbbE L_0 \Pr \paren*{ N < \frac n2 } + \sum_{k \geq 2^{-1} n} \abs{\BbbE L_n - \BbbE L_k} \Pr \paren{N = k}.
\end{equation*}
Note that the first term on the right of \eqref{eq:Ln-mean-eq1} is monotonically decreasing in $k$, since both $\tau_k / k$ and $\zeta_k$ monotonically decrease in $k$ for fixed $n$.
Therefore for $k \geq 2^{-1} n$,
\begin{equation*}
    \abs{\BbbE L_n - \BbbE L_k}
    \leq \frac{C_* \tau_k}{2^{-1} n} \cdot (2 \zeta_k)^p \abs{n-k}
     + \BbbE L_0 \sum_{l > 2^{-1} n} h_l.
\end{equation*}
Since $\BbbE \abs{N-n} \leq \sqrt n$ and $\BbbE L_0 = \abs{x-y}^p$, after expanding $\tau_k$ and $\zeta_k$ we have
\begin{equation*}
    \frac{\abs{\BbbE L_n - \BbbE L_N }}{(n f(z))^{(1-p)/d} \abs{x-y}}
    \leq O\paren*{\frac{((n f(z))^{1/d} \abs{x-y})^{\alpha p}}{\sqrt n}} + \frac{ \Pr\paren{N < 2^{-1} n} + \sum h_l }{((n f(z))^{1/d} \abs{x_n-y_n})^{1-p}},
\end{equation*}
where the summation $\sum h_l$ is still for $l > 2^{-1} n$.  The first term on the right decays to zero since $n^{\alpha p/d} < \sqrt{n}$ and $\abs{x-y} < 2R_1$.  The second term also decays to zero since, while the denominator has at most polynomial decay in $n$, $\sum h_l$ in the numerator has exponential decay in $n^\alpha$ by \autoref{thm:euc-shortlink} and \autoref{thm:Ln-size}, and $\Pr\paren{N < 2^{-1} n}$ in the numerator has exponential decay in $n$ by the Chernoff bound \autocite[Theorem~9.3]{Billingsley:1995a}.
Note that $\cX_N \cap B(z; R_2)$ is identically distributed as $\cH_{nf(z)} \cap B(z; R_2)$, and the proposition follows since the difference of $\BbbE L_N = \BbbE L(x,y; \cH_{nf(z)} \cap B(z; R_2))$ from $C(d,p)$ is less than $\vareps$ for sufficiently large $n$ by \autoref{thm:poisson-mean-cor}.
\end{proof}

\appendix
\section*{Appendix}
Here we show how \eqref{eq:main-cor} can be derived from \autoref{thm:euc-unif-local}.

As we did in the proof of \autoref{thm:poisson-mean-cor}, scale the space by factor of $\lambda^{-1/d}$ with choice $\lambda = t^d$ to obtain
\begin{equation}
\label{eq:appendix-scale}
  \Pr \paren*{ \abs*{ \frac1t \cL_1(0,te_1) - C(d,p) } > s }
  = \Pr \paren*{ \abs*{ \frac{\cL_\lambda(0,e_1)}{\lambda^{(1-p)/d}} - C(d,p) } > s }.
\end{equation}

Let $N$ be a Poisson random variable with mean $\lambda V_d R^d$, where $V_d$ denotes the volume of a unit ball.
Fix $0 < \delta < 1$ and $0 < \vareps < C(d,p)$.
Suppose that
\begin{enumerate}[label=(\roman*)]
  \item \label{it:append-1} $(1-\delta) \BbbE N \leq N \leq (1+\delta) \BbbE N$, and
  \item \label{it:append-2} the events $E_n(\vareps)$ from \autoref{thm:euc-unif-local} occur for all $n$ in the range $(1-\delta) \BbbE N \leq n \leq (1+\delta) \BbbE N$.
\end{enumerate}
For \autoref{thm:euc-unif-local}, choose $z = 0 \in \reals^d$, $b = 0$, and pick $R$ so that $R > 8 (C(d,p) + \vareps) / 5 (C(d,p) - \vareps)$ and $R > 4$ so that $e_1 \in B(z; R/4)$.

Note that $\cH_\lambda \cap B(z; R)$ may be realized as $\cX_N$ where $X_1,X_2,\dots$ are uniform i.i.d.\ random variables in $B(z; R)$, and thus $L(0,e_1; \cH_\lambda \cap B(z; R)) = L(0,e_1; \cX_N)$.
Therefore under the assumptions \autoref{it:append-1} and \autoref{it:append-2}, we have
\begin{equation*}
  \abs*{ \frac{L(0,e_1; \cH_\lambda \cap B(z; R))}{(Nf(z))^{(1-p)/d}} - C(d,p) } \leq \vareps,
\end{equation*}
and
\begin{equation*}
  \frac{L(0,u; \cH_\lambda \cap B(z; R))}{(N f(z))^{(1-p)/d}} \geq \frac58 R (C(d,p) - \vareps) > C(d,p) + \vareps
\end{equation*}
for all $u \notin B(0; R)$ by the choice of $R$.
Therefore the path $\cL_\lambda(0,e_1) = L(0,e_1; \cH_\lambda)$ is contained in $B(z; R)$ and
\begin{equation*}
  \abs*{ \frac{\cL_\lambda(0,e_1)}{(N f(z))^{(1-p)/d}} - C(d,p) } \leq \vareps.
\end{equation*}
Note that $f(z) = (V_d R^d)^{-1}$.
Using the condition \autoref{it:append-1} and $\BbbE N = \lambda V_d R^d$, we have
\begin{equation*}
  \frac{C(d,p) - \vareps}{(1+\delta)^{(p-1)/d}}
  \leq \frac{\cL_\lambda(0,e_1)}{\lambda^{(1-p)/d}}
  \leq \frac{C(d,p) + \vareps}{(1-\delta)^{(p-1)/d}}.
\end{equation*}
We can choose $\vareps$ and $\delta$ small enough so that
\begin{equation*}
  \abs*{ \frac{\cL_\lambda(0,e_1)}{\lambda^{(1-p)/d}} - C(d,p) } \leq s.
\end{equation*}

In summary, the probability in \eqref{eq:appendix-scale} is bounded above by the probability that either \autoref{it:append-1} or \autoref{it:append-2} does not occur.

From \autoref{thm:euc-unif-local}, there exists a constant $\theta_4 > 0$ such that $1 - \Pr(E_n(\vareps)) \leq \exp(-\theta_4 (n f(z))^\alpha)$ for all sufficiently large $n$.
By \autoref{it:append-2}, $n f(z) \geq (1-\delta) \lambda$.
Denote by $H_n$ the event of \autoref{it:append-1}.
Then for sufficiently large $\lambda = t^d$, \eqref{eq:appendix-scale} is bounded above by 
\begin{gather*}
  (1 - \Pr(H_n)) + 2 \delta \lambda V_d R^d \exp(-\theta_4 ((1-\delta) \lambda)^\alpha)
  \\
  = (1 - \Pr(H_n)) + \exp(-\theta_4 (1-\delta)^\alpha t^{d/(d+2p)} + O(\log t)).
\end{gather*}
Note that $1 - \Pr(H_n)$ is exponentially small in $n$ by the Chernoff bound \autocite[Theorem~9.3]{Billingsley:1995a}, so that the tail is dominated by the second term.

\printbibliography
\end{document}